\documentclass[onefignum,onetabnum]{siamart171218}
\usepackage[utf8]{inputenc}
\usepackage{amsmath}
\usepackage{amsfonts}
\usepackage{amssymb}
\usepackage{graphicx}
\usepackage{algorithmic}
\usepackage[sort&compress, numbers]{natbib}
\usepackage{perpage}
\MakePerPage{footnote}

\headers{Optimized mixing by cutting-and-shuffling}{L. D. Smith, P. B. Umbanhowar, J. M. Ottino, and R. M. Lueptow}

\title{Optimized mixing by cutting-and-shuffling\thanks{Submitted to the editors \today.
\funding{P.B. Umbanhowar was partially supported by the National Science Foundation Contract No. CMMI-1435065.}}}

\author{Lachlan D. Smith\thanks{Department of Chemical and Biological Engineering, Northwestern University, Evanston, IL 60208, USA 
  (\email{lachlan.smith@northwestern.edu}, \email{jm-ottino@northwestern.edu}, \email{r-lueptow@northwestern.edu}).}
\and Paul B. Umbanhowar\thanks{Department of Mechanical Engineering, Northwestern University, Evanston, IL 60208, USA 
  (\email{umbanhowar@northwestern.edu}).}
\and Julio M. Ottino\footnotemark[2] \footnotemark[3] \thanks{The Northwestern Institute on Complex Systems (NICO), Northwestern University, Evanston, IL 60208, USA.}
\and Richard M. Lueptow\footnotemark[2] \footnotemark[3] \footnotemark[4]
}

\graphicspath{{figures/}}

\begin{document}

\maketitle


\begin{abstract}
Mixing by cutting-and-shuffling can be understood and predicted using dynamical systems based tools and techniques. In existing studies, mixing is generated by maps that repeat the same cut-and-shuffle process at every iteration, in a ``fixed'' manner. However, mixing can be greatly improved by varying the cut-and-shuffle parameters at each step, using a ``variable'' approach. To demonstrate this approach, we show how to optimize mixing by cutting-and-shuffling on the one-dimensional line interval, known as an interval exchange transformation (IET). Mixing can be significantly improved by optimizing variable protocols, especially for initial conditions more complex than just a simple two-color line interval. While we show that optimal variable IETs can be found analytically for arbitrary numbers of iterations, for more complex cutting-and-shuffling systems, computationally expensive numerical optimization methods would be required. Furthermore, the number of control parameters grows linearly with the number of iterations in variable systems. Therefore, optimizing over large numbers of iterations is generally computationally prohibitive. We demonstrate an \emph{ad hoc} approach to cutting-and-shuffling that is computationally inexpensive and guarantees the mixing metric is within a constant factor of the optimum. This \emph{ad hoc} approach yields significantly better mixing than fixed IETs which are known to produce weak-mixing, because cut pieces never reconnect. The heuristic principles of this method can be applied to more general cutting-and-shuffling systems.
\end{abstract}



\section{Introduction}

Cutting-and-shuffling has recently been shown to be an effective method for mixing granular materials \cite{Sturman2008, Sturman2012, Juarez2010, Juarez2012, Christov2014, Park2016, Park2017, Smith2017BSTresonances, Smith2017BSTCS+SF, Zaman2017} and fluids \cite{Smith2016discdef, Smith2017LSID, Jones1988, Schonfeld2004, Schultz2014, Hunt2008, Boujlel2016}. However, previous studies only consider systems where the same cut-and-shuffle action is repeated at every step, for example, cutting a deck of cards at the 10th and 20th cards, and swapping the top pile with the bottom pile. Repeating this cut-and-shuffle could potentially transform an unmixed deck of cards (e.g., organized by suit) into a ``mixed'' state (no two cards of the same suit together). However, it seems likely that a mixed state would be reached faster if the locations of the cuts were chosen strategically at each step. Here we explore how this can be accomplished.

Consider the simplest form of cutting-and-shuffling for a continuous system: cutting a 1D line interval into $m$ pieces, and permuting them to reassemble the line. An example of this process, known as an Interval Exchange Transformation (IET), is shown in \cref{fig:IET}(a), where the line, initially consisting of equal length black, gray and white segments, is cut into four pieces, and the permutation $\pi = 3142$ represents the rearrangement order of the four cut pieces. The notation $\pi=3142$ means that the third cut piece moves to the first position, the first piece moves to the second position, fourth piece moves to third position, and second piece moves to fourth position. In the second iteration, $N=2$, the same cut locations and permutations are used. We refer to IETs with the same cut locations and permutation at every iteration as ``fixed,'' since the same action is repeated; these can be thought of as analogues of time-periodic fluid flows. Fixed IETs like that in \cref{fig:IET}(a) can generate good mixing.

While fixed IETs can produce mixing, i.e.\ complete homogenization of material given infinite time \cite{Keane1975, Keane1977, Viana2006}, mixing by cutting-and-shuffling occurs at best at a polynomial rate, which is significantly slower than the exponential mixing rates produced by chaotic systems \cite{Veech1978, Katok1980, Masur1982, Avila2007, Sturman2012, Christov2011}. To demonstrate, consider the time-space plots shown in \cref{fig:IET}(b,c). Here, the iterates of the 1D line interval under the cut-and-shuffle operation are stacked so that the initial condition ($N=0$) is at the top and iteration $N=65$ is at the bottom (many intermediate iterates have been omitted). When the lengths of the cut pieces are rationally dependent [\cref{fig:IET}(b)], or the permutation is reducible\footnote{A permutation $\pi$ is said to be irreducible if applying $\pi$ to any of the subsets $\{1\}, \{1,2\},\dots,\{1,2,\dots,L-1\}$ does not yield a permutation of just the elements of that subset. A permutation that is not irreducible is termed reducible.}, the colors do not mix well. In fact, the colors reassemble to their initial condition for the example in \cref{fig:IET}(b). When the lengths of the cut pieces are rationally independent, and the permutation is irreducible and not a rotation, the IET satisfies the Keane minimality condition \cite{Keane1975, Keane1977, Viana2006} and yields good mixing of the colors [\cref{fig:IET}(c)]. However, additional parametric freedom in the mixing process can be introduced by allowing the cut locations and/or permutation to change at every iteration. We call this approach ``variable,'' and it is analogous to time-dependent fluid flow, which changes at every instant. In this study we consider variable IETs such that the permutation is fixed, but the cut locations can vary. Intuitively, the additional control freedom at each iteration should enable faster and better mixing when using variable IETs rather than fixed IETs.

Finding the best variable protocol is an optimal control problem. That is, the goal is to minimize or maximize an objective function (the degree, rate, or efficiency of mixing) over the parameter space. Optimal mixing has been studied for time-dependent laminar fluid flows \cite{Cortelezzi2008, Froyland2016, Froyland2017, Mathew2007, Rallabandi2017, Vikhansky2002, Mitchell2017}, in which mixing occurs by stretching-and-folding. Here we consider optimal mixing by cutting-and-shuffling. There are a number of differences between mixing produced by cutting-and-shuffling compared to traditional mixing produced by stretching-and-folding. To start with, cutting-and-shuffling results in discontinuous interfaces when applied to a scalar field, even when the initial scalar field is smooth. This means mixing can be highly sensitive to parameter choices, even for a single iteration. In addition, standard mixing metrics such as the mix-norm \cite{Mathew2005}, intensity of segregation \cite{Danckwerts1952}, and interface length \cite{Leong} do not vary smoothly, or even continuously, across the parameter space, which poses unique challenges in applying optimal control theory to cutting-and-shuffling.

\begin{figure}[tbp]
\centering
\includegraphics[width=0.45\columnwidth]{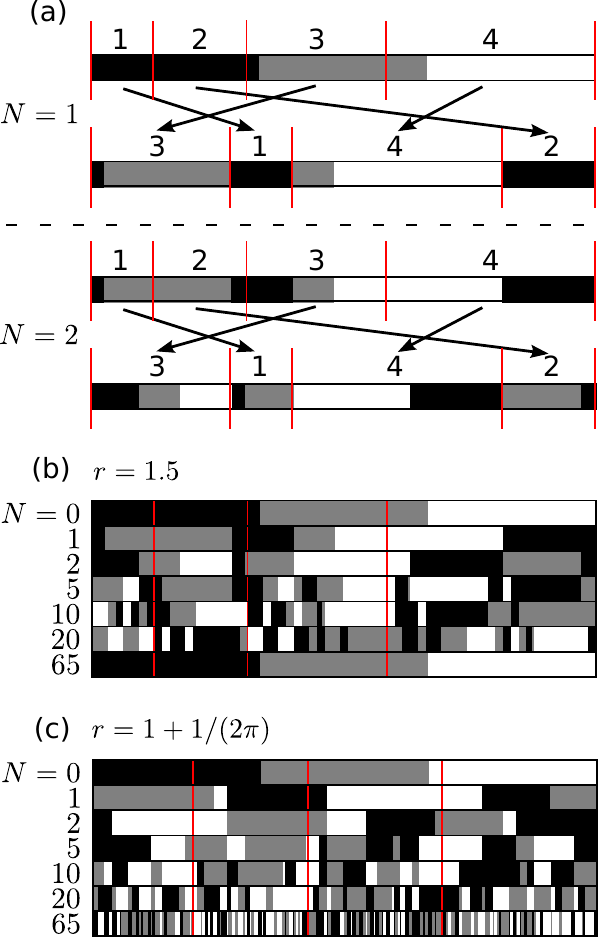}
\caption{Cutting-and-shuffling a line with a fixed IET. (a)~Two iterations of a fixed IET. The line, initially consisting of single black, gray, and white segments, is cut into four pieces, and rearranged according to the irreducible permutation $3142$ at each iteration. Cut locations are shown by red vertical lines. The cut pieces have lengths $x, rx, r^2x, r^3x$, where $r=1.5$ is the ratio of successive cut piece lengths, and $x=(r-1)/(r^4-1)$, as used by Krotter \emph{et al.} \cite{Krotter2012}. (b,c)~Iterates of the line interval are stacked vertically to create a space-time plot. With the same permutation, rational values of $r$ produce (b)~periodic dynamics, while irrational values of $r$ generate (c)~weak mixing. For fixed IETs like these, the cut locations and permutation are the same at each iteration and can produce mixing [e.g., (c)], but, if the cuts were strategically chosen at each iteration, mixing could be improved.}
\label{fig:IET}
\end{figure}

In \cref{sec:metrics}, we discuss different metrics that can be used to quantify mixing by IETs. A new metric is introduced that quantifies how well the colored segments are broken up into smaller segments and whether the different colors are evenly distributed along the line. We demonstrate that neither weak-mixing fixed IETs, nor a naive, \emph{ad hoc}, variable approach -- cutting the longest segments of each color in half at each iteration -- yield optimal mixing.

In \cref{sec:analytic_optima}, we demonstrate how to analytically find optimal variable IETs with fixed permutation but variable cut locations. We show that optimal mixing can only be produced when the permutation takes a specific form, and demonstrate how to find cut locations that achieve optimal mixing. Then in \cref{sec:time_dependent-vs-time-periodic} we show that optimal variable IETs produce significantly better mixing than general fixed IETs, and when the initial condition is more complex, there is more improvement gained by using an optimal variable IET compared to both random fixed IETs and optimal fixed IETs.

In \cref{sec:long-time_optimization}, we discuss strategies for mixing by cutting-and-shuffling over many iterations. For general cutting-and-shuffling systems, both fixed and variable, finding protocols that optimize mixing over many iterations can be computationally expensive. One alternative option to achieve good, though sub-optimal, mixing is to use geometric properties of fixed piecewise isometries. For instance, fixed IETs that satisfy the Keane minimality condition \cite{Keane1975, Viana2006, Avila2007}, such as in \cref{fig:IET}(c), yield weak-mixing. However, for initial conditions like the three-color initial condition in \cref{fig:IET}, it is likely that a weak-mixing fixed IET will produce slow mixing, since identically colored segments sometimes reconnect after the cut pieces are shuffled. This means colored segments will not always be broken into smaller and smaller pieces. Another alternative is to optimize over short time-horizons. We demonstrate this approach for IETs using the computationally inexpensive \emph{ad hoc} method introduced in \cref{sec:metrics}, which is equivalent to a one-iteration time-horizon optimization. This \emph{ad hoc} method produces significantly better mixing than weak-mixing fixed IETs, because segments of the same color never reconnect. Similar heuristic approaches could be applied to more general cutting-and-shuffling systems, cutting the largest unmixed regions at each iteration.

\section{Optimal variable cutting-and-shuffling}

\subsection{Mixing metrics} \label{sec:metrics}

Optimization of mixing requires a mixing metric. Qualitatively, for IETs with more than two differently colored segments, mixing is effective if the segments are broken up into many smaller segments \textit{and} the different colors are evenly distributed along the line. Thinking of a deck of cards, the suits would be considered well mixed if any set of four consecutive cards contains one card of each suit.

For a line interval consisting of $m$ colored segments, $\mathcal{I}_1,\dots,\mathcal{I}_m$, with lengths $|\mathcal{I}_i|$ as shown in \cref{fig:metrics}(a), one mixing metric used previously is the longest segment length
\begin{equation}
U = \max_{1\leq i \leq m} |\mathcal{I}_i|,
\end{equation}
termed the percent unmixed \cite{Krotter2012, Yu2016}. For example, $U=|\mathcal{I}_5|$ in \cref{fig:metrics}(a), and $U=1/9$ in \cref{fig:metrics}(b). Note that we assume periodic boundary conditions at the ends of the line, so segments of the same color at the start and end of the line are connected. Using periodic boundaries means that $U$ is invariant under rotations (left or right shifts). The same periodic boundary conditions have also been used in past studies \cite{Ashwin2002, Sturman2012, Froyland2016}. Finding IETs that minimize $U$ ensures that all the segment lengths are small, suggesting good mixing, but $U$ does not consider whether the colors are distributed evenly along the line. For the deck of cards analogy, $U$ is minimized if every set of two consecutive cards have different suits, but this can be achieved by riffling the hearts and diamonds, and riffling the spades and clubs. Even though $U$ is minimized, the deck would not be considered mixed, as all the red cards are separate from all the black cards.

To measure the evenness of color distribution, we calculate the longest distance between segments of the same color. Letting $d_{ij}$ represent the distance between the closest edges of the $j$-th and $(j+1)$-th segments of the $i$-th color, where $i$ spans the number of different colors and $j$ spans the number of segments with the $i$-th color, the evenness metric is
\begin{equation}
D = \max_{i,j} d_{ij}.
\end{equation}
In \cref{fig:metrics}(a), $D=d_{32}$, and in \cref{fig:metrics}(b), $D=2/9$. When the colors are evenly distributed $D$ is small, and when they are clustered together $D$ is large. Note that if there are only two colors, then $D$ and $U$ are identical, otherwise they are generally different. 

\begin{figure}[tbp]
\centering
\includegraphics[width=0.6\columnwidth]{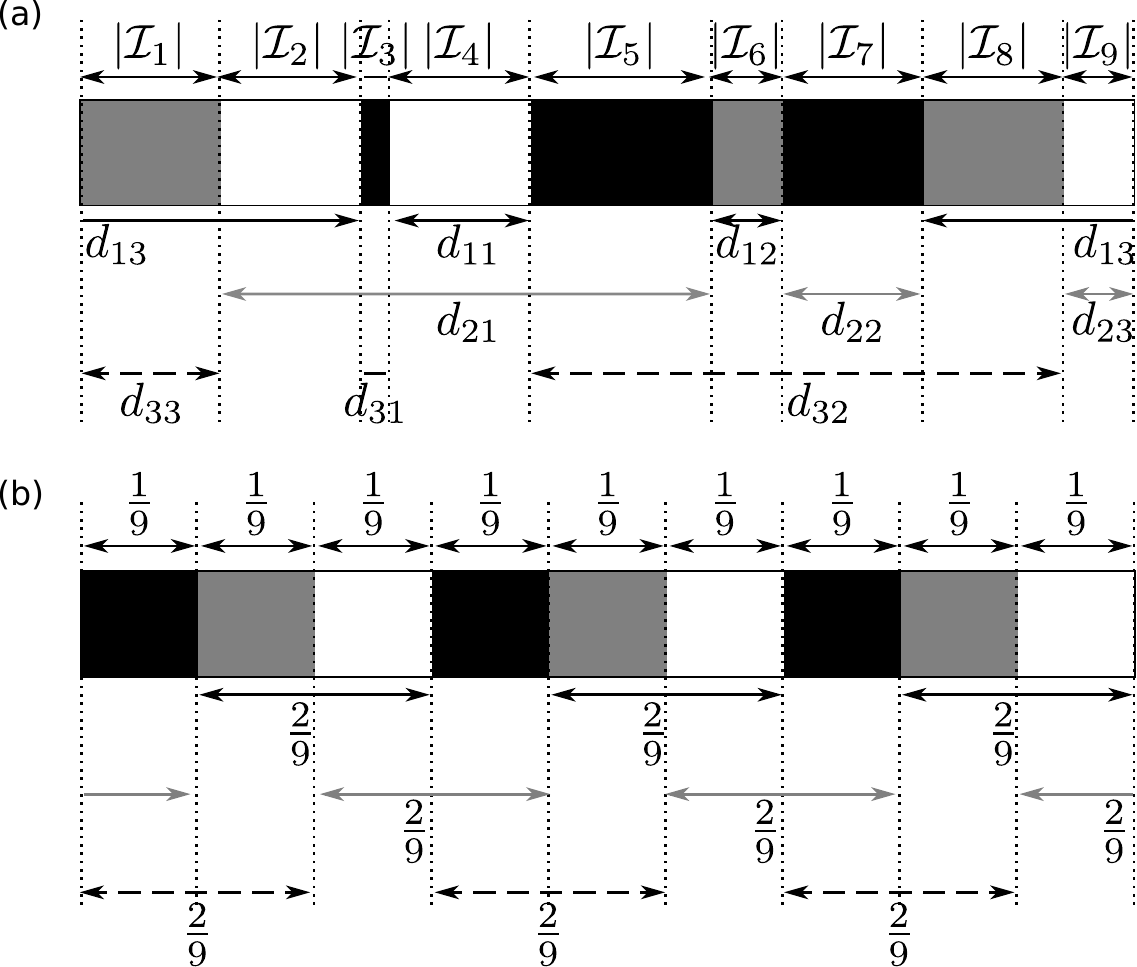}
\caption{Different metrics evaluate different aspects of mixing colors on the line. (a)~$|\mathcal{I}_i|$ measures the length of the $i$-th segment, whereas $d_{ij}$ measures the distance between the closest edges of the $j$-th and $(j+1)$-th segments of the $i$-th color. The longest segment is $U=|\mathcal{I}_5|$, and the longest distance between segments of the same color is $D=d_{32}$. (b)~The optimal case, where all $|\mathcal{I}_i|$ are equal and all $d_{ij}$ are equal, such that $U=1/9$ and $D=2/9$.}
\label{fig:metrics}
\end{figure}

For an IET that starts with $k$ differently colored segments and uses a length $L$ permutation, at most $L-1$ new segments are created per iteration. Therefore, there are at most $N(L-1) + k$ total segments after $N$ iterations, and $U$ has a minimum of $(N(L-1)+k)^{-1}$, when all the segments are equal length. We scale $U$ by its minimum value,
\begin{equation}
 \hat{U}(N,L,k) = \left( N(L-1) + k \right) U,
 \end{equation} 
to measure how well an IET has cut the colored segments into smaller pieces compared to the known optimum. Values of $\hat{U}$ greater than $1$ correspond to suboptimal cutting of the colored segments into smaller pieces. Similarly, $D$ is minimized when there are $k-1$ differently colored segments between each pair of same-colored segments [e.g., \cref{fig:metrics}(b)], and each segment has length $(N(L-1)+k)^{-1}$, so we scale $D$:
\begin{equation}
\hat{D}(N,L,k) = \frac{N(L-1) + k}{k-1} D.
\end{equation}
Values of $D$ greater than $1$ correspond to uneven distributions of the different colors.

For mixing colors on the line, there are two competing interests -- small segment lengths ($U$) and evenly distributed colors ($D$) -- and an optimal protocol for one metric is not necessarily an optimal for the other. We use simple linear scalarization to handle this multi-objective optimization task by defining the metric, $\Phi=(\hat{U} + \hat{D})/2$ as the total measure of mixing. Values of $\Phi$ greater than $1$ mean that either the segments are not cut into small pieces or the colors are not evenly distributed, or both. When segments are equal in length and colors are uniformly distributed, $\Phi=\hat{U}=\hat{D} = 1$.

Considering the fixed IET in \cref{fig:IET}(c), even though it is known to be weak-mixing (i.e.\ given an infinite number of iterations it will completely homogenize the colors), we see that after 65 iterations mixing is suboptimal. Only $148$ segments are produced, compared to the maximum $N(L-1)+k=198$. In addition, the segments are clearly not equal in length, quantified by $\hat{U}=3.4$, meaning the longest segment is more than three times longer than in an optimal case. Similarly, the colors are not well distributed, $\hat{D}=8.3$. There appears to be a darker region near the middle of the line, where black segments are clustered together, and a lighter region near the right-hand end, where white segments are clustered together. In this case, $\Phi=5.9$ is significantly higher than an optimal variable IET for which $\Phi=1$.

\begin{figure}[tbp]
\centering
\includegraphics[width=0.5\textwidth]{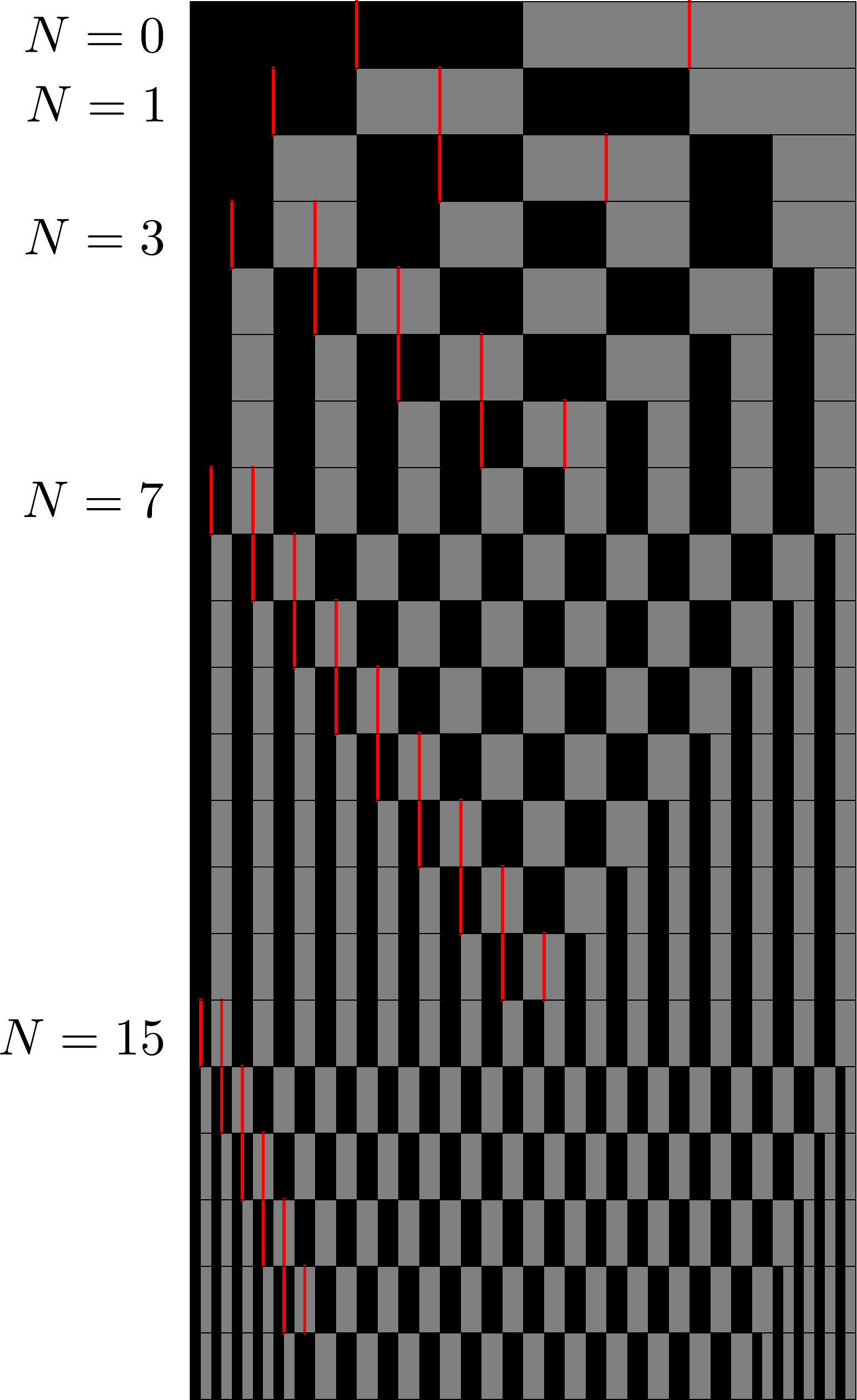}
\caption{Mixing the two-color initial condition (top row) using the permutation $132$ and the \emph{ad hoc} method, where the two cuts (red) are made in the middle of the longest black and gray segments at each iteration. The bottom row at $N=20$ corresponds to $\Phi=21/16 \approx 1.3$, whereas $\Phi=1$ at $N=1,3,7,15$.}
\label{fig:ad_hoc_mixing}
\end{figure}

For a given number of cuts and a given number of initial colors, optimal mixing corresponds to evenly distributed segments of equal length like the case shown in \cref{fig:metrics}(b). The question is: Can variable IETs be found that achieve optimal mixing? A naive, \emph{ad hoc}, approach, is to cut the longest segments in half at each iteration. This is demonstrated for the two-color initial condition in \cref{fig:ad_hoc_mixing}, where the longest black and gray segments at each iteration are cut in half, and the pieces are rearranged according to the permutation $\pi=132$. For such a simple approach, this method performs remarkably well. Whenever $N=2^i-1$, the segments all have equal length, and the maximum number of segments, $N(L-1)+k=2N+2$, has been created, meaning $\Phi=1$ ($N=1,3,7,15$ in \cref{fig:ad_hoc_mixing}). However, for any other value of $N$, mixing is suboptimal. For instance, $\Phi=21/16\approx 1.3$ at $N=20$ (the bottom row of \cref{fig:ad_hoc_mixing}). In the next section, we demonstrate how to find optimal variable IETs ($\Phi=1$).


\subsection{Finding optimal variable IETs} \label{sec:analytic_optima}

Consider the two-color initial condition with the first half of the line colored black, and the second half gray [top row of \cref{fig:matching_edges}(a)], and a variable IET $T_{L,\pi}$ that cuts the line into $L$ pieces, with the $j$-th cut in the $N$-th iteration located at $c_{Nj}$, and rearranges the cut pieces according to a permutation $\pi$. After cutting, but before rearrangement, there are $L$ cut pieces, each with a left edge and a right edge. After rearrangement, new segments and new interfaces can only be created if left and right edges of different colors join. Hence, at most $L-1$ new segments/interfaces can be created (taking into account the existing interface in the initial condition between the start and end of the line). For example, the permutation $132$ has three cut pieces, and creates at most two new segments per iteration, as demonstrated in \cref{fig:matching_edges}(a). Therefore, after $N$ iterations there can be at most $N(L-1)$ new segments. However, not all variable IETs produce $L-1$ new segments per iteration, the trivial example being any IET that uses the identity permutation, which cuts but does not shuffle, leaving the line unchanged regardless of the cut locations. Similarly, rotation permutations, i.e.\ those of the form $R(i)=i+r\!\! \mod L$, do not shuffle, and so cannot produce new segments \cite{Yu2016}. The following proposition gives the necessary and sufficient conditions for a variable IET to be able to produce $L-1$ new segments per iteration.

\begin{figure}[tbp]
\centering
\includegraphics[width=0.9\columnwidth]{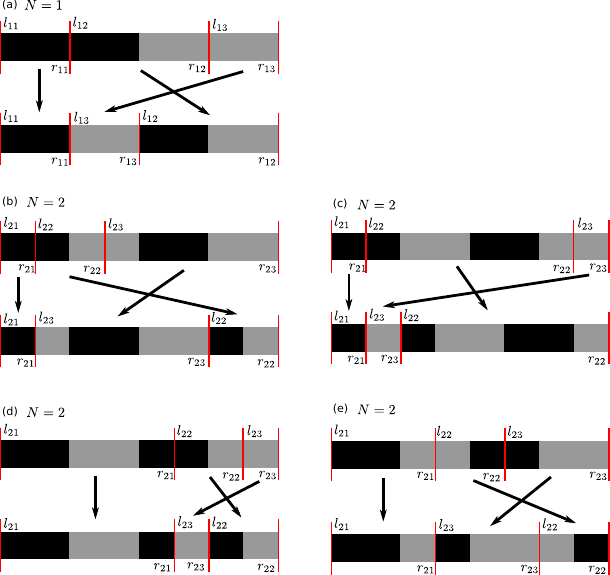}
\caption{Demonstration of \cref{prop:2} for the two-color initial condition [top row of (a)] and permutation $\pi=132$. The colors of the left edges of cut pieces are labelled $l_{N1},l_{N2},l_{N3}$, and the colors of the right edges of cut pieces are labelled $r_{N1},r_{N2},r_{N3}$. (a)~In the first iteration, $N=1$, cutting within the black segment and within the gray segment yields the maximum number of segments, $2N+2=4$. This is because $r_{1j}=l_{1,j+1}$, $j=1,2$, i.e.\ cuts are within segments, and $r_{1,\pi(j)}\neq l_{1,\pi(j+1)}$, i.e.\ each cut creates a new black-gray interface after rearrangement by $\pi$. (b--d)~For $N=2$, as long as the first cut is located within a black segment and the second cut is located within a gray segment, then the maximum number of segments, $2N+2=6$, is produced. This is because $l_{2j}=l_{1j}$ and $r_{2j}=r_{1j}$ for $j=1,2,3$, and so $r_{2j}=l_{2,j+1}$, $j=1,2$, and $r_{2,\pi(j)}\neq l_{2,\pi(j+1)}$. (e)~For $N=2$, locating the first cut within a gray segment and the second cut within a black segment does not yield the maximum number of segments. There are still only four segments (due to periodic boundary conditions the black segments at the start and end of the line form a single segment).}
\label{fig:matching_edges}
\end{figure}

\begin{proposition} \label{prop:valid_perms}
For the left-right two-color initial condition, there exists a variable IET $T_{L,\pi}$ that creates $L-1$ new segments per iteration if and only if $L$ is odd and $\pi$ is of the form 
\begin{equation} \label{eq:valid_perms}
\pi=1\, \pi_2(n+2) \left[ \prod_{i=2}^k \pi_1(i) \pi_2(i+n+1) \right] (n+1) \left[\prod_{i=k+1}^n \pi_1(i) \pi_2(i+n+1) \right],
\end{equation}
where $n=(L-1)/2$, $\pi_1$ is a permutation of the set $\{2,\dots,n\}$, $\pi_2$ is a permutation of the set $\{n+2,\dots,L=2n+1\}$, and $1 \leq k \leq n$, or $\pi$ is a rotation of a permutation of this form.
\end{proposition}

By a rotation of a permutation $\pi\in S_n$, we mean a permutation $\tau\in S_n$ of the form $\tau(i) = \pi (i+r\!\! \mod L)$ for some $r$. For example, the rotations of the permutation $\pi=1324$ are $3241$, $2413$, and $4132$. The rotated permutation $\tau$ is the composition $\pi \circ R$, where $R$ is the rotation permutation $R(i) = i+r\!\! \mod L$.

As an example of \cref{prop:valid_perms}, for $L=7$, choose $\pi_1 (2\, 3)=(3\, 2)$, $\pi_2 (5\, 6\, 7) = (6\, 5\, 7)$, and $k=2$, then 
\begin{equation} 
\pi = 1\, \pi_2(5) \pi_1(2) \pi_2(6)\, 4\, \pi_1(3) \pi_2(7) = 1635427.
\end{equation}
By \cref{prop:valid_perms}, there exist cut locations such that the corresponding variable IET, $T_{L,\pi}$, creates $L-1=6$ new segments per iteration. The same is true for any rotation of $\pi$, e.g.\ $6354271$.

\begin{corollary}
For $L=2n+1$ odd, there are $(n!)^2(2n+1)$ permutations $\pi$ for which there is a variable IET $T_{L,\pi}$ that can produce $L-1$ new segments per iteration.
\end{corollary}

\begin{proof}
There are $(n-1)!$ choices for the permutation $\pi_1$, $n!$ choices for the permutation $\pi_2$, $n$ choices for $k$, and $2n+1$ rotations of a permutation of length $2n+1$. Therefore, the number of rotations of permutations of the form \cref{eq:valid_perms} is
\begin{equation}
n(n-1)!n!(2n+1)=(n!)^2(2n+1).
\end{equation}
\end{proof}

We prove \cref{prop:valid_perms} using the following results.

\begin{proposition} \label{prop:2}
For the left-right two-color initial condition and a permutation $\pi$ with length $L$, if there exist cut locations $c_{1j}$, $j=1,\dots,L-1$, such that cutting and rearranging the initial condition according to $\pi$ yields $L-1$ new segments, then there exist cut locations $c_{Nj}$ for all $N>1$ and $j=1,\dots,L-1$ such that the corresponding variable IET, $T_{L,\pi}$, creates $L-1$ new segments per iteration.
\end{proposition}

\begin{proof}
Suppose there exist cut locations $c_{1j}$, $j=1,\dots,L-1$, such that cutting and rearranging the initial condition according to $\pi$ produces $L-1$ new segments. Let $l_{N1},\dots,l_{NL}$ denote the colors of the left edges of the cut pieces for the $N$-th iteration, and $r_{N1},\dots,r_{NL}$ denote the colors of the right edges of the cut pieces. From the initial condition, we know $l_{11}$ is black (the left edge of the line) and $r_{1L}$ is gray (the right edge of the line), and since $L-1$ new segments are created after rearrangement by $\pi$, each left edge must join with a differently colored right edge, i.e.\ $r_{1,\pi(j)}\neq l_{1,\pi(j+1)}$ for $j=1,\dots,L$. After rearrangement there are two cases, either the left-most segment is black or gray. If the left-most segment is black, for the second iteration choose the cut locations $c_{2j}$, $j=1,\dots,L-1$, such that the left and right edges satisfy $l_{2j}=l_{1j}$ and $r_{2j}=r_{1j}$, respectively, for $j=1,\dots,L$. It follows that after rearrangement by $\pi$, adjacent segments satisfy $r_{2,\pi(j)}=r_{1,\pi(j)}\neq l_{1,\pi(j+1)}=l_{2,\pi(j+1)}$, so each left edge joins with a differently colored right edge, and $L-1$ new segments have been created. 

If the left-most segment is gray, for the second iteration choose the cut locations $c_{2j}$, $j=1,\dots,L-1$, such that the left and right edges satisfy $l_{2j}\neq l_{1j}$ and $r_{2j}\neq r_{1j}$, i.e.\ the left and right edges of cut pieces have the opposite colors to the first iteration. Again, $\pi$ must join every left edge with a differently colored right edge, since that was the action in the first iteration. Choosing the opposite color sequence is equivalent to switching the two colors (black segments become gray, and vice versa), then performing the cut and shuffle, and then switching the two colors back.

In successive iterations, as long as the cut locations $c_{Nj}$ are chosen such that the left edges $l_{Nj}$ and right edges $r_{Nj}$ of cut pieces have the colors $l_{11},\dots,l_{1L}$ and $r_{11},\dots,r_{1L}$ (or their opposites if the first colored segment at the $N$-th iteration is gray), then $L-1$ new segments will be created.
\end{proof}

For example, for the two-color initial condition, variable IETs with permutation $\pi=132$ create two new segments in the first iteration, $N=1$, as long as the first cut $c_{11}$ is in the black segment and the second cut $c_{12}$ is in the gray segment, as shown in \cref{fig:matching_edges}(a). In successive iterations, $N\geq 2$, two new segments are created if the first cut $c_{N1}$ is in a black segment, and the second cut $c_{N2}$ is in a gray segment, as shown in \cref{fig:matching_edges}(b--d) for $N=2$. Note that the cuts do not need to occur in the first black and gray segments, e.g., \cref{fig:matching_edges}(c,d).

\Cref{prop:2} shows that the creation of new segments only depends on the colors of the edges of the cut pieces, and does not depend on the colors within the interior of each cut piece.

From \cref{prop:2}, it suffices to consider only whether there exist cut locations for a permutation that can create $L-1$ new segments in the first iteration. 

\begin{proposition} \label{prop:rotations}
For a permutation $\pi$, if there exist cut locations $c_{Nj}$ such that the corresponding variable IET $T_{L,\pi}$ yields $L-1$ new segments per iteration, then for any rotation $\tau$ of $\pi$, there exist cut locations $c'_{Nj}$ such that $T_{L,\tau}$ yields $L-1$ new segments per iteration.
\end{proposition}

\begin{proof}
The action of the IET with permutation $\tau(i) = \pi (i+r \!\!\mod L)$, i.e.\ $\tau$ is a rotation of $\pi$, is to perform the rearrangement of the cut pieces by $\pi$, and then shift all the segments along the line. Therefore, if $L-1$ new segments are created in the first iteration by $T_{L,\pi}$ with cuts $c_{1j}$, then the same number of segments is created by $T_{L,\tau}$ with cuts $c'_{1j}=c_{1j}$. From \cref{prop:2} there exist cut locations $c'_{Nj}$ for all $N>1$ such that $T_{L,\tau}$ produces $L-1$ new segments per iteration. 
\end{proof}

Therefore, we need only consider permutations $\pi$ such that $\pi(1)=1$, as all other permutations are rotations of these.

\begin{proposition} \label{prop:odd_L_cuts_inside_segments}
For the left-right two-color initial condition and a permutation $\pi$, there exist cut locations $c_{Nj}$ such that the corresponding variable IET $T_{L,\pi}$ yields $L-1$ new segments per iteration only if $L$ is odd, and at each iteration $n=(L-1)/2$ cuts are within black segments, and $n$ cuts are within gray segments.
\end{proposition}

\begin{proof}
A new black-gray interface can only be created when the right edge of a black cut piece joins with the left edge of a gray cut piece, and vice versa for gray-black interfaces. Therefore, the maximal number of new segments, $L-1$, can only be created at each iteration if the number of black right edges matches the number of gray left edges, and the number of gray right edges matches the number of black left edges. Otherwise, right and left edges of the same color would have to join, meaning a new interface would not be created. Furthermore, cuts must not occur at existing interfaces, because that cannot lead to a net increase in interfaces. This means cuts must occur within colored segments, and for each black/gray right edge there must also be a left edge with the same color. To create $L-1$ new interfaces, the number of cuts within black segments must match the number of cuts within gray segments at each iteration. Otherwise, there is an imbalance of either black or gray edges. When $L$ is odd ($L-1$ cut locations), this can be achieved by having $(L-1)/2$ cuts in black segments, and $(L-1)/2$ cuts in gray segments. However, when $L$ is even, there must either be more cuts within black segments or more cuts within gray segments, resulting in an imbalance in the number of black edges and gray edges.
\end{proof}

We can now return to the proof of \cref{prop:valid_perms}. From the previous three propositions we can restrict our attention to the first iteration of variable IETs $T_{L,\pi}$, such that $\pi(1)=1$, $L=2n+1$ is odd, and cuts $c_{1j}$ are located such that $n$ are in the black segment and $n$ are in the gray segment. After cutting (but before rearranging) the left-right initial condition, there are $2n+1$ cut pieces, call them $p_1,\dots,p_{2n+1}$. The first $n$ cut pieces, $p_1,\dots, p_{n}$, are black, and we label them $b_1,\dots, b_{n}$, i.e.\ $b_i=p_i$. The final $n$ cut pieces, $p_{n+2},\dots,p_{2n+1}$, are gray, and we label them $g_1,\dots, g_{n}$, i.e.\ $g_i=p_{i+n+1}$. There is also one middle piece, $p_{n+1}$, that is black on the left and gray on the right. Since $\pi(1)=1$, after rearrangement by $\pi$, the first piece remains in the first position. For optimal mixing, the next piece must be gray, so it can be any one of the $g_i$. Ignoring the black-gray piece, $p_{n+1}$, for now, the next piece must be black, i.e.\ one of the $b_i$, and we simply alternate between black and gray pieces, to get a sequence
\begin{equation}
b_1 g_{\pi_2^*(1)} b_{\pi_1(2)} g_{\pi_2^*(2)} \dots b_{\pi_1(n)} g_{\pi_2^*(n)},
\end{equation}
where $\pi_1$ is a permutation of $\{2,\dots,n\}$, and $\pi_2^*$ is a permutation of $\{1,\dots,n\}$, representing the permutations of the black and gray pieces respectively. Since $b_i=p_i$ and $g_i=p_{i+n+1}$, this is equivalent to the sequence
\begin{multline}
p_1 p_{\pi_2^*(1)+n+1} p_{\pi_1(2)} p_{\pi_2^*(2)+n+1} \dots p_{\pi_1(n)} p_{\pi_2^*(n)+n+1} \\
 =  p_1 p_{\pi_2(n+2)} p_{\pi_1(2)} p_{\pi_2(n+3)} \dots p_{\pi_1(n)} p_{\pi_2(2n+1)},
\end{multline}
where $\pi_2(i) = \pi_2^*(i-n-1)+n+1$ is the permutation of $\{n+2,\dots,2n+1\}$ obtained by conjugating $\pi_2^*$ with the shift $i\mapsto i+n+1$. Returning to the black-gray piece, $p_{n+1}$, we can insert it immediately after any of the gray pieces, i.e.\ those of the form $p_{\pi_2(i)}$, to get $L-1$ new segments. In such a case, the permutation of the pieces, i.e.\ the sequence of indices, is of the form 
\begin{equation} 
\pi=1\, \pi_2(n+2) \left[ \prod_{i=2}^k \pi_1(i) \pi_2(i+n+1) \right] (n+1) \left[\prod_{i=k+1}^n \pi_1(i) \pi_2(i+n+1) \right],
\end{equation}
where $1\leq k \leq n$ indicates which gray piece the black-gray piece follows. This proves \cref{prop:valid_perms}.


Note that some reducible permutations, such as $132$, can produce optimal mixing in the two-color variable case [\cref{fig:segment_rescaling}(b)], but they cannot even produce weak mixing in the fixed case, as they do not satisfy the Keane minimality condition \cite{Keane1975}. 


Now that we know which permutations can produce the maximal number of new colored segments, the question is where the cuts need to be located within the black and gray segments to achieve optimal mixing, such that each segment in the final state has the same length. Consider the case $L=3$ and $\pi = 132$; at each iteration we cut the first black segment in half and the first gray segment in half at cut locations $c_{Nj}$, as shown in \cref{fig:segment_rescaling}(a). By \cref{prop:valid_perms} and the proof of \cref{prop:2}, the maximum number, $L-1=2$, new segments will be created at each iteration. However, the colored segments will not have equal lengths, so the protocol does not mix optimally. This can be overcome by rescaling the segments after a desired number of iterations so that they all have the same length, and then iterating backward to find where the cuts need to be located. In this case, the rescaling occurs between the bottom rows of \cref{fig:segment_rescaling}(a,b), and by iterating backward (upward) in \cref{fig:segment_rescaling}(b), the optimal cut locations $c_{Nj}'$ are found, such that $\Phi=1$ at $N=2$. This rescaling procedure can be repeated for any variable IET, and so cut locations that achieve $\Phi=1$ can be found for any variable IET $T_{L,\pi}$ that satisfies \cref{prop:valid_perms}.

\begin{figure}[tbp]
\centering
\includegraphics[width=\columnwidth]{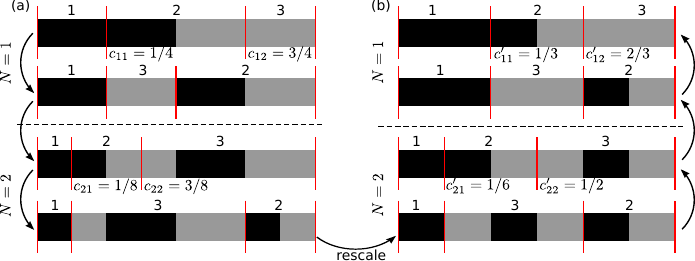}
\caption{Finding optimal cut locations for the two-color initial condition and variable IET with permutation $132$. Red lines indicate cut locations. (a)~Cutting at the midpoint of the first black and gray segments at each iteration yields the maximum number of segments, $2N+2$, but for $N\geq 2$ segment lengths are unequal. Cut locations are denoted $c_{Nj}$, where $1\leq j \leq L-1$ is the location of the $j$-th cut in the $N$-th iteration. Optimal cut locations can be found by rescaling the bottom row of (a) so that all segments have the same length, as shown in the bottom row of (b), then iterating backward (upward) using the same sequence of cuts as in (a) to find the optimal cut locations, $c_{Nj}'$.}
\label{fig:segment_rescaling}
\end{figure}

Cutting the first black segment and first gray segment at each iteration, as in \cref{fig:segment_rescaling}, gives one optimal mixing protocol, but there are many more. For instance, before cutting at $N=2$, there are two black segments and two gray segments. Instead of performing the cuts in the first black and first gray segments, we can also achieve optimal mixing by cutting in the first black segment and in the second gray segment [\cref{fig:matching_edges}(c)], or in the second black segment and in the second gray segment [\cref{fig:matching_edges}(d)]. As long as the first cut is in a black segment and the second cut is in a gray segment to its right, then optimal mixing will be achieved. We code the different optima as pairs $(i,j)$, where $i$ represents which black segment is cut (e.g.\ $i=1$ if the first black segment is cut) and $j$ represents which gray segment is cut (e.g.\ $j=2$ if the second gray segment is cut). At the $N$-th iteration, there are $N$ black segments and $N$ gray segments, so $1\leq i \leq j \leq N$. Hence, at the $N$-th iteration, there are $\left(\genfrac{}{}{0pt}{}{N+1}{2}\right)$ distinct pairs of optimal cut locations, where $\left(\genfrac{}{}{0pt}{}{a}{b}\right)$ denotes the binomial coefficient. Over the full $N$ iterations, the total number of distinct sets of optimal cut locations equals
\begin{equation} \label{eq:number_cut_locations}
\prod_{i=1}^N \left(\genfrac{}{}{0pt}{}{N+1}{2}\right)  = \frac{\left(N! \right)^2 \left( N+1 \right)}{2^N} .
\end{equation}

This same approach can be used to find optimal mixing protocols for arbitrary numbers of iterations, $N$, and cut pieces, $L=2n+1$. At the $N$-th iteration, there are $(N-1)(L-1)+2=2[n(N-1)+1]$ total segments, half black and half gray. Assuming the first segment is always black (or re-coloring as needed), and that the first $n$ cuts are always within black segments, and the final $n$ cuts are within gray segments, then by the proof of \cref{prop:2}, optimal mixing ($\Phi=1$) is guaranteed (after rescaling segment lengths). Like in the case $L=3$ above, we code the optimal cut locations by $2n$-tuples $(i_1,\dots,i_n,j_1,\dots,j_n)$, where the $i$'s represent which black segments are cut (e.g., $i_3=2$ means the third cut occurs in the second black segment), and the $j$'s represent which gray segments are cut (e.g., $j_5=3$ means the fifth gray segment cut occurs in the third gray segment). Since the cuts in black segments must come first, the $i$'s and $j$'s satisfy $1\leq i_1 \leq \dots\leq i_n \leq j_1 \leq \dots \leq j_n \leq n(N-1)+1$, and hence there are $\left(\genfrac{}{}{0pt}{}{n(N+1)}{2n}\right)$ possibilities. However, in some cases there are even more optima, because it is not always necessary to perform the first $n$ cuts within black segments and the final $n$ cuts within gray segments. For example, for variable IETs using the permutation $14325$, it can be shown that after the first iteration we can achieve optimal mixing by locating the cuts in the order black, black, gray, gray (BBGG) as described above; or the orders BGGB, GBBG, and GGBB. The orders BGBG and GBGB cannot produce optimal mixing (see \cref{app:color_order} for full details). In the case $L=3$ with permutation $132$, the first cut must always be in a black segment and the second in a gray segment. Hence, \cref{eq:number_cut_locations} accounts for all possibilities (\cref{app:color_order}).

\subsection{Extension to more than two colors} \label{sec:multi-color_extension}

What is important is that we have a formulaic approach to find variable IETs that produce optimal mixing. Even though most of the results in this section use a two-color initial condition, many of the ideas can be extended to initial conditions with more colors. The main challenge when extending to more than two colors is that the metrics $U$ and $D$ are not equivalent. Optimal mixing is not simply a matter of producing the maximal number of segments like in the two-color case; the ordering of the colored segments is equally important to minimize $D$. Therefore, it is more difficult to find necessary conditions for mixing to be optimal. However, it is relatively easy to find permutations and cut locations that produce optimal mixing for more than two colors. We follow essentially the same construction as the two-color case. For $k$ colors, $C_1,\dots,C_k$, $L-1$ must be a multiple of $k$ in order to produce the same number of each colored segment at each iteration (equivalent to the condition that $L$ must be odd in \cref{prop:odd_L_cuts_inside_segments}). In the first iteration, we make $(L-1)/k$ cuts within each of the $k$ segments (like \cref{prop:odd_L_cuts_inside_segments}). For the permutation $\pi$, we choose $\pi(1)=1$ (\cref{prop:rotations} is valid for any number of colors), so that after rearranging, the first colored segment is still color $C_1$. Then, through trial-and-error or solving a system of linear equations with the colors of the left and right edges of cut pieces as variables (see \cref{app:color_order}), we find the permutations that yield the sequence $C_1,\dots,C_k$ repeated $1+(L-1)/k$ times, i.e.\ those permutations that yield $(L-1)/k$ new sequences of the colors. This is equivalent to repeating the black-gray pair for the two-color case. In successive iterations, as long as cut segments have the same colors as in the first iteration (i.e.\ the first $(L-1)/k$ cuts are within segments with color $C_1$, the next $(L-1)/k$ cuts are within segments with color $C_2$, and so on), then the IET will yield the sequence $C_1,\dots,C_k$ repeated $1+N(L-1)/k$ times, where $N$ is the number of iterations. We then use the same rescaling process demonstrated in \cref{fig:segment_rescaling} to find cut locations such that all the segments have the same length after a desired number of iterations, which means the IET mixes optimally, i.e.\ $\Phi=1$. 

Consider the three-color initial condition ($k=3$) in \cref{fig:segment_rescaling_3_color}. To produce optimal mixing, $L-1$ must be a multiple of $3$. For $L=4$, one cut is made within each of the black, gray and white segments, as shown in the top row of \cref{fig:segment_rescaling_3_color}(a). We assume that $\pi(1)=1$, so that after rearrangement by $\pi$, the first piece, $p_1$, does not move. Since $p_1$ has a black right edge, to obtain the sequence black-gray-white repeated $1+(L-1)/k=2$ times, the next piece must have a gray left edge, meaning it must be $p_3$. Now, the right edge of $p_3$ is white, so the next piece must have a black left edge, meaning it must be $p_2$. Lastly, the fourth piece, $p_4$, remains in place. Therefore, the permutation $\pi = 1324$ yields two repeating black-gray-white sequences. In subsequent iterations, the sequence black-gray-white will be repeated $N+1$ times as long as the first cut is within a black segment, the second cut is within a gray segment, and the third cut is within a white segment. This is demonstrated in \cref{fig:segment_rescaling_3_color}(a), where the cuts are made at the midpoints of the first black, gray and white segments. However, for $N=2$ in \cref{fig:segment_rescaling_3_color}(a), mixing is suboptimal because the segment lengths are unequal. Optimal cut locations are found using the same rescaling process demonstrated in \cref{fig:segment_rescaling}, i.e.\ segments in the bottom row of \cref{fig:segment_rescaling_3_color}(a) are rescaled to obtain the bottom row of \cref{fig:segment_rescaling_3_color}(b), then the IET is iterated backward.

\begin{figure}[tbp]
\centering
\includegraphics[width=\columnwidth]{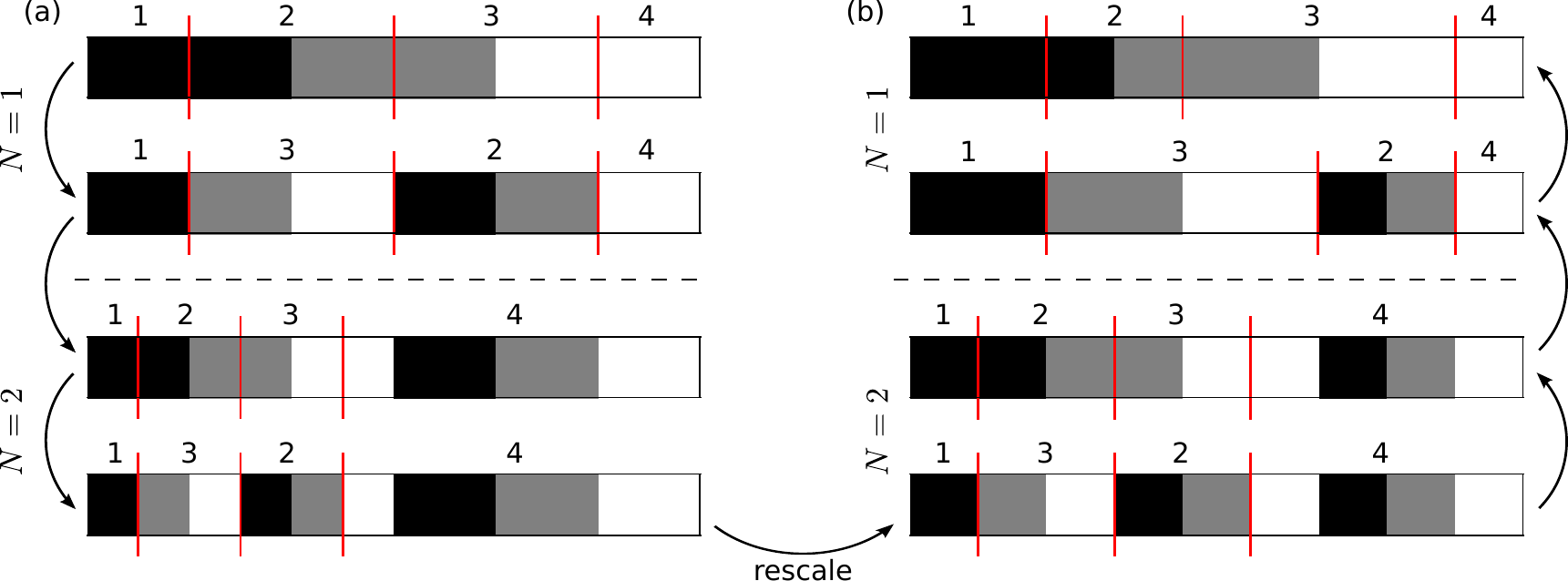}
\caption{Finding optimal cut locations for the three-color initial condition and variable IET with permutation $1324$. Red lines indicate cut locations. (a)~Cutting at the midpoint of the first black, gray and white segments at each iteration yields the maximum number of segments, $N(L-1)+k=3N+3$, but for $N\geq 2$ segment lengths are unequal. Optimal cut locations can be found by rescaling the bottom row of (a) so that all segments have the same length, as shown in the bottom row of (b), then iterating backward (upward) using the same sequence of cuts as in (a) to find the optimal cut locations.}
\label{fig:segment_rescaling_3_color}
\end{figure}

The above approach provides a relatively easy way to find permutations and cut locations that mix optimally for any number of colors. In fact, the permutations found using this approach are the \emph{only} permutations that can mix optimally. The only permutations that this method would not capture are those that permute the repeating sequence of colors, for instance changing from repeating black-gray-white to repeating black-white-gray. Consider the first iteration, $N=1$, in \cref{fig:segment_rescaling_3_color}(a), and suppose that the repeating sequence of colors changed from black-gray-white to black-white-gray. Since the piece $p_1$ has a black left edge, after rearrangement it must be preceded by a piece with a gray right edge, meaning it must be preceded by $p_2$. However, $p_2$ followed by $p_1$ results in the sequence black-gray-black, and so colors cannot be evenly distributed along the line (each pair of black segments must have a gray segment and a white segment between them). Therefore, for the case with three colors and $L=4$, the repeating sequence of colors must remain black-gray-white.

More generally, since cuts must occur within the colored segments, in the first iteration there must be $k-1$ cut pieces that each have exactly two colors, $C_i$ on the left and $C_{i+1}$ on the right, for $i=1,\dots,k-1$. For example, with three colors (black-gray-white), there must be a black-gray piece and a gray-white piece ($p_2$ and $p_3$ at each iteration in \cref{fig:segment_rescaling_3_color}). For an optimally mixing permutation, after the first iteration (and every iteration) each set of $k$ colored segments must have one of each color (evenly distributed colors), and so the same sequence of colors $C_{a_1},\dots,C_{a_k}$ must repeat itself, where $a_1,\dots,a_k$ is a permutation of $1,\dots,k$. Assuming $a_1=1$ (we can always start with the first color in the initial condition), there is a cut piece with two colors, $C_1$ on the left and $C_2$ on the right, so $C_2$ must always follow $C_1$, and $a_2=2$. Similarly, there is a cut piece with $C_2$ on the left and $C_3$ on the right, so $C_3$ must always follow $C_2$, and $a_3=3$. Continuing this process for all the pieces with two colors yields $a_i=i$ for all $i=1,\dots,k$, i.e.\ the ordering of the colors cannot change, and so our method captures \emph{all} optimally mixing permutations.

Now that we know how to find optimally mixing variable IETs, the question is: How much better is the mixing they produce compared to fixed IETs?

%
%
%

\section{Variable vs. fixed cutting-and-shuffling} \label{sec:time_dependent-vs-time-periodic}

Intuitively, the added parametric freedom of variable protocols should enable significantly improved mixing. In this section we compare mixing produced by optimal variable IETs, like those discussed in the previous section, with mixing produced by fixed IETs with random cut locations and fixed IETs with optimally chosen cut locations. Considering two, three, and four colors in the initial condition, we show that in all cases optimal variable IETs produce significantly better mixing than random fixed IETs and that the degree of improvement increases with the number of iterations, $N$. Furthermore, with more colors in the initial condition, optimal variable IETs improve mixing more than both random and optimal fixed IETs.

First, consider mixing the two-color initial condition in the top row of the space-time plot in \cref{fig:segment_rescaling}. For permutations with $L=3$, from \cref{prop:valid_perms} there are only three permutations that can achieve optimal mixing in the variable case, $132$, $321$, and $213$. The permutations $132$ and $213$ are both reducible, and hence do not mix in the fixed case (the IET is periodic with period equal to two). On the other hand, the permutation $321$ is irreducible, and hence can at least achieve weak-mixing in the fixed case (when the two cut locations are chosen to satisfy the Keane minimality condition). However, weak-mixing only guarantees mixing over infinite iterations, and we are more interested in optimizing mixing over finite numbers of iterations. For fixed IETs with permutation $321$, we calculate $\Phi$ across the cut location parameter space, $0<c_1<c_2<1$, (sampling on a grid with a spacing of $5\times 10^{-3}$ in each direction) The results are shown in \cref{fig:mix_metric_distros} for $N=2$, $4$, $6$, and $8$ iterations, with the darkest color intensity indicating the value of the mixing metric $\Phi$ closest to the optimum of $\Phi=1$. The average value of $\Phi$, $\Phi_{\text{ave}}$, across the $(c_1,c_2)$ parameter space is the expected, or ``typical,'' degree of mixing for randomly chosen fixed cut locations. As the number of iterations $N$ increases, $\Phi_{\text{ave}}$ grows approximately linearly, shown by the dotted black curve in \cref{fig:periodic_metric_comparison}(a). This means that on average, the mixing quality, compared to the variable optimum ($\Phi=1$), becomes worse as the number of iterations increases. To show a typical space-time plot for an average fixed IET, we arbitrarily select a particular IET such that $\Phi \approx \Phi_{\text{ave}}$ for each $N$, which is shown in the first column of \cref{fig:periodic_optima_2_colors}. As $N$ increases, there is a small improvement in mixing, i.e.\ there are more black and gray segments, and the length of the longest colored segment, $U$, decreases. However, the colored segments vary substantially in their lengths. Therefore, for the average ``typical'' fixed IET, the mixing metric $\Phi$, which is normalized by the optimal variable IET, increases rapidly as $N$ increases. In contrast, for the optimal variable IETs (third column of \cref{fig:periodic_optima_2_colors}), the number of segments grows faster [as $N(L-1)+k=2N+2$], and $U$ decreases more rapidly [as $(2N+2)^{-1}$], indicating all segments have uniform length.

Now we compare optimal variable IETs to optimal fixed IETs, again with the two-color initial condition and $L=3$. Optimal fixed IETs are found for each $N$ by progressively refining the sample grid in the $(c_1,c_2)$ parameter space around a minimum value of $\Phi$, $\Phi_{\text{min}}$, shown as green `$\times$'s in \cref{fig:mix_metric_distros}.\footnote{Due to symmetry of the parameters about the line $c_2=1-c_1$ there are two minima. We find the one such that $c_2<1-c_1$.} For the two-color initial condition, in contrast to the average, $\Phi_{\text{min}}$ remains close to $1$ [dotted black curve in \cref{fig:periodic_metric_comparison}(b)], indicating that the optimal fixed IET mixes almost as well as the optimal variable IET. This is demonstrated in the second column of \cref{fig:periodic_optima_2_colors}. The optimal fixed IETs (second column) yield the maximum number of segments, $N(L-1)+k=2N+2$, and the segments are relatively uniform, with the uniformity of the segments improving with $N$. As a result, $\Phi$ approaches $1$ as $N$ increases.


\begin{figure}[tbp]
\centering
\includegraphics[width=\columnwidth]{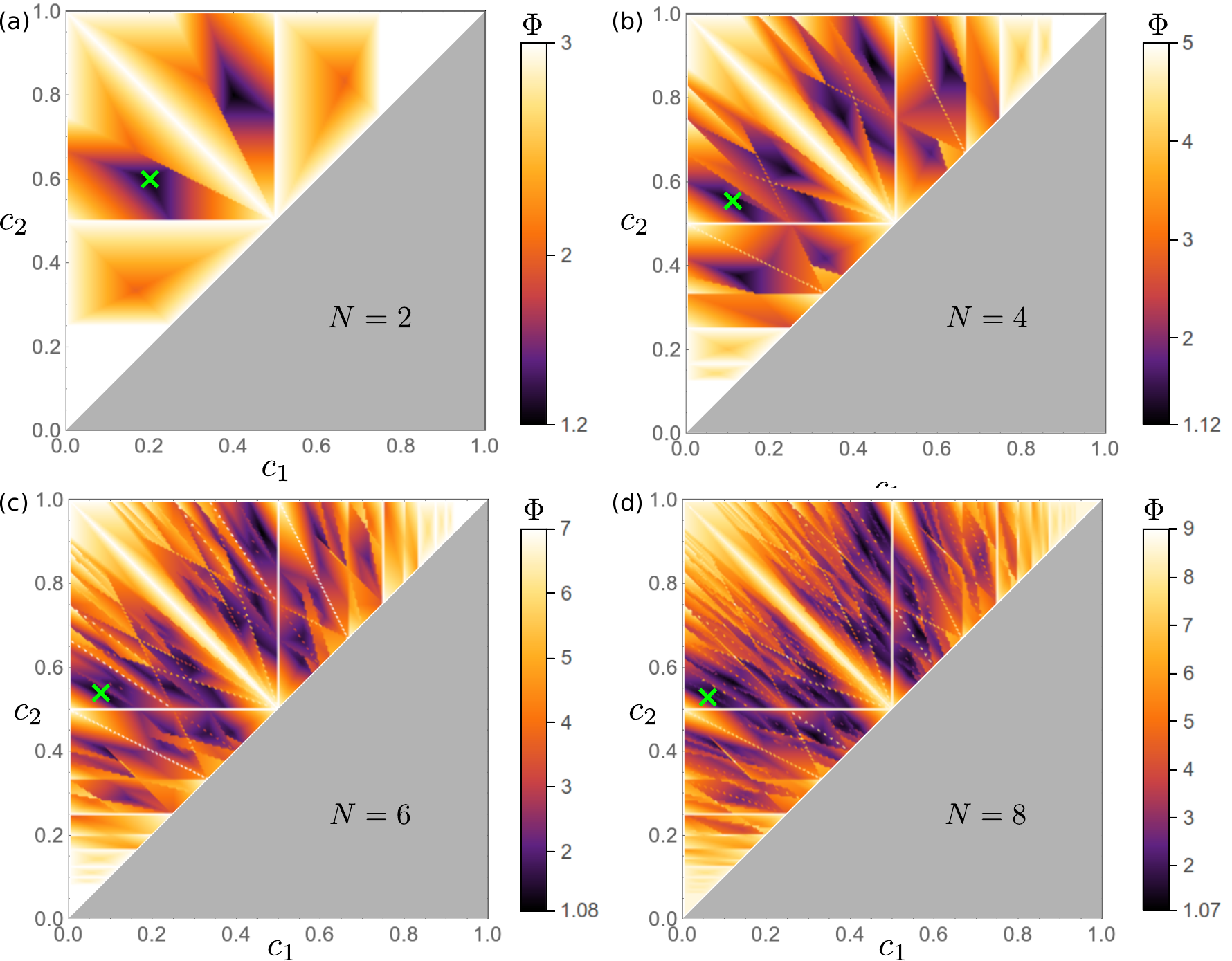}
\caption{Mixing metric $\Phi$ across the cut location parameter space $0<c_1<c_2<1$ for fixed IETs using the permutation $321$ and the two-color initial condition shown in the top rows of \cref{fig:segment_rescaling}. Cut locations with minimum $\Phi$ (optimal mixing) are marked by a green `$\times$'. Due to symmetry through the line $c_2=1-c_1$, there is also a second minimum (not indicated). Note that each plot has a different range for $\Phi$, spanning the extremes of $\Phi$.}
\label{fig:mix_metric_distros}
\end{figure}

\begin{figure}[tbp]
\centering
\includegraphics[width=0.55\columnwidth]{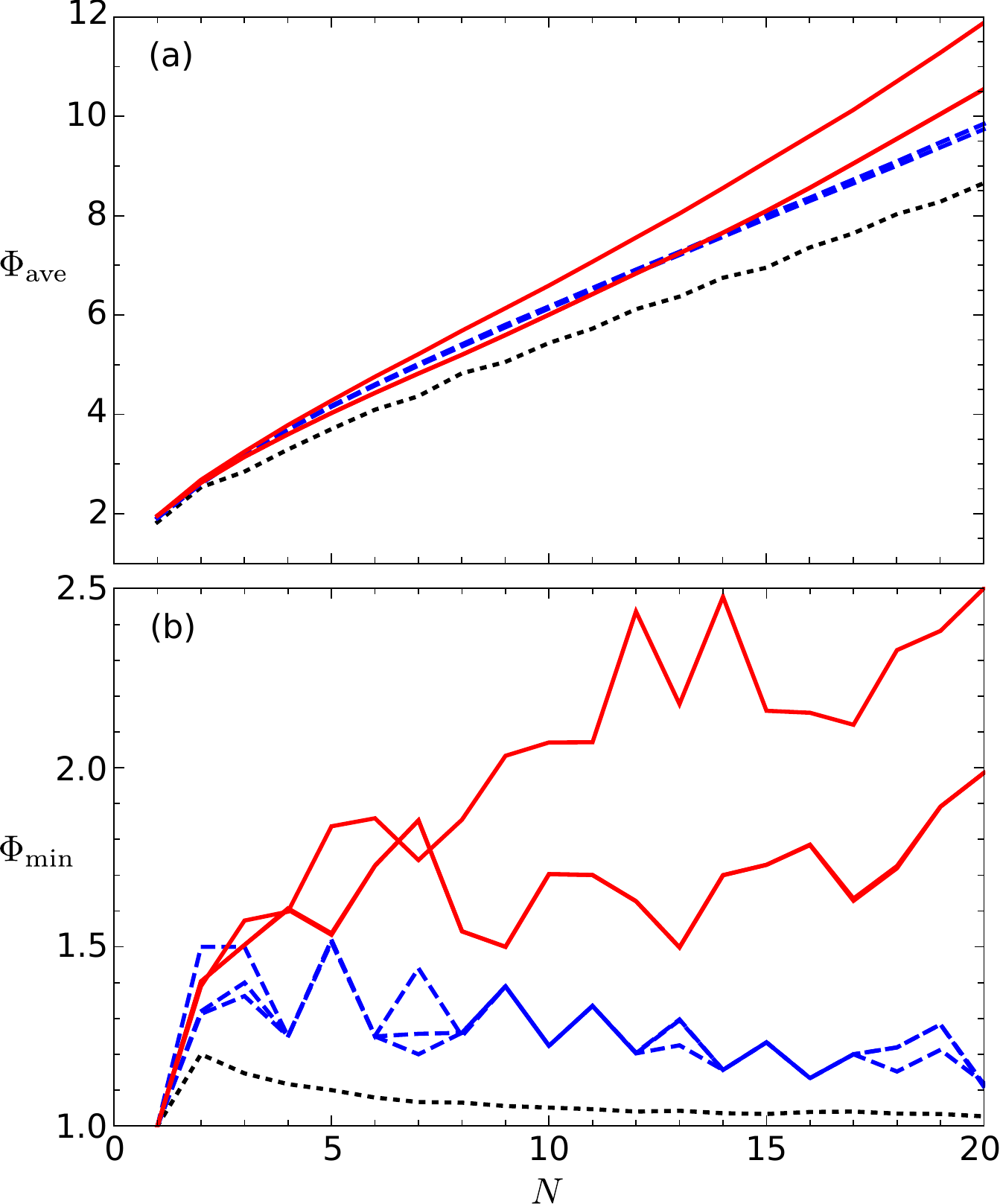}
\caption{(a)~Average and (b)~minimum of the mixing metric $\Phi$ across the $0<c_1<\dots<c_{L-1}<1$ parameter space for fixed IETs with different irreducible permutations and initial conditions. Dotted black: two color initial condition and the $321$ permutation. Dashed blue: three color initial condition, and the permutations $3241$, $2413$, and $4132$ (in (a) the curves all overlap, and in (b) the curves overlap in several regions). Red: four color initial condition and the permutations $52413$, $35241$, and $41352$ (the curves for the first two permutations are the same due to a reflection symmetry). Optimal variable IETs have $\Phi=1$.}
\label{fig:periodic_metric_comparison}
\end{figure}



\begin{figure}[tbp]
\centering
\includegraphics[width=\columnwidth]{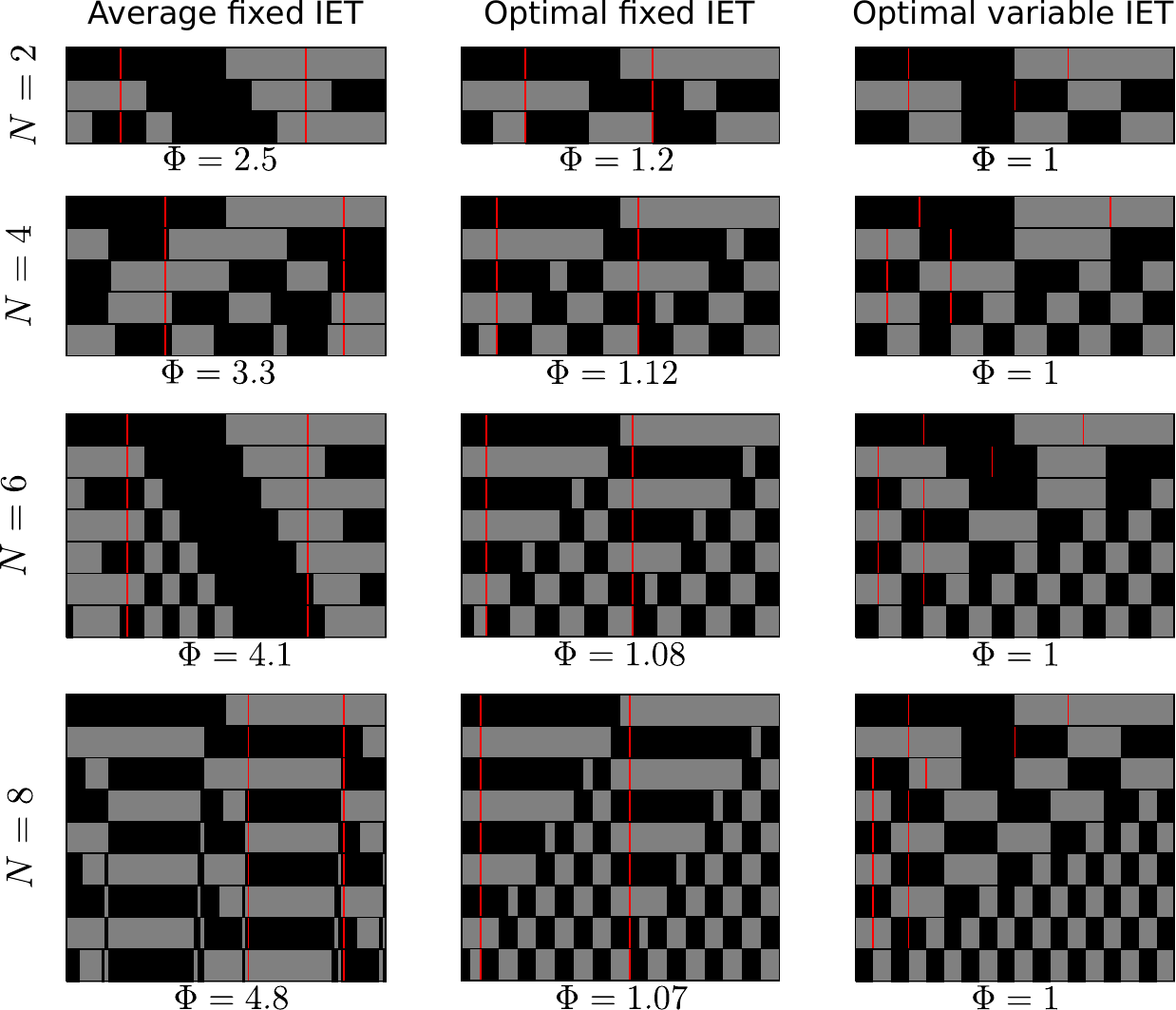}
\caption{Space-time plots for IETs with the two-color initial condition using the permutation $\pi=321$ for $N=2$, $4$, $6$, and $8$. Cut locations are shown by red lines. First column: fixed IETs with $\Phi\approx\Phi_{\text{ave}}$. Second column: optimal fixed IETs with $\Phi=\Phi_{\text{min}}$. Cut locations correspond to the green $\times$'s in \cref{fig:mix_metric_distros}. Third column: optimal variable IETs ($\Phi=1$) obtained using the methods described in \cref{sec:analytic_optima}.}
\label{fig:periodic_optima_2_colors}
\end{figure}

We perform similar analyses for three-color and four-color initial conditions. As discussed in \cref{sec:multi-color_extension}, for the three-color initial condition, $L-1$ must be a multiple of three to achieve optimal mixing. Thus, $L=4$ is the minimum permutation length. In this case, the only permutations that can achieve optimal mixing are the rotations of $1324$, of these three are irreducible, namely $3241$, $2413$, and $4132$. Similarly, for the four-color initial condition, $L-1$ must be a multiple of four to achieve optimal mixing. For $L=5$, the only permutations that can achieve optimal mixing are the rotations of $13524$, of these three are irreducible, namely $35241$, $52413$, and $41352$. For each initial condition, each irreducible permutation, and each $N$, we find the average, $\Phi_{\text{ave}}$, and minimum, $\Phi_{\text{min}}$, values of $\Phi$ for fixed IETs by sampling the parameter space $0<c_1<\dots<c_{L-1}<1$, similar to the approach to generate \cref{fig:mix_metric_distros} for two colors and $L=3$. As with the two-color initial condition, $\Phi_{\text{ave}}$ grows approximately linearly as $N$ increases for both the three-color [dashed blue curves in \cref{fig:periodic_metric_comparison}(a), all three curves overlap] and four-color [solid red curves in \cref{fig:periodic_metric_comparison}(a), two of the three curves overlap due to a reflection symmetry] initial conditions, meaning mixing using random fixed cuts, compared to the variable optimum ($\Phi=1$), becomes progressively worse as $N$ increases. This is demonstrated in the space-time plots for fixed IETs with $\Phi \approx \Phi_{\text{ave}}$ in the first columns of \cref{fig:periodic_optima_3_colors,fig:periodic_optima_4_colors}. As $N$ increases, the number of segments increases, the length of the longest segment decreases, and the colors become more evenly distributed along the line, indicating improved mixing. However, when compared to the optimal variable IETs (third columns of \cref{fig:periodic_optima_3_colors,fig:periodic_optima_4_colors}) it is clear that at each $N$, the average fixed IETs, have fewer segments, the segments are longer, and the colors are not as evenly distributed, leading to large values of $\Phi$. Since the discrepancy in mixing quality between average fixed IETs and optimal variable IETs becomes greater as $N$ increases, $\Phi_{\text{ave}}$ increases with $N$ [\cref{fig:periodic_metric_comparison}(a)].

Comparing $\Phi_{\text{ave}}$ for the different initial conditions, \cref{fig:periodic_metric_comparison}(a) shows that when $N>14$, $\Phi_{\text{ave}}$ is greater when there are more colors in the initial condition (the curves are ordered vertically according to the number of colors in the initial condition). This means that when there are more colors in the initial condition, mixing using random fixed cut locations becomes worse relative to the optimum ($\Phi=1$). The same vertical ordering of curves occurs when considering optimal fixed IETs (those with $\Phi=\Phi_{\text{min}}$) as well, shown in \cref{fig:periodic_metric_comparison}(b). Therefore, when there are more colors in the initial condition, mixing using optimal fixed cut locations also becomes worse relative to the optimum ($\Phi=1$). While the optimal fixed IETs for the two-color initial condition could achieve near optimal values of $\Phi$, the optimal fixed IETs for the three-color and four-color initial conditions perform significantly worse. This is demonstrated by the space-time plots for optimal fixed IETs in the second columns of \cref{fig:periodic_optima_3_colors,fig:periodic_optima_4_colors}. In each case the maximum number of segments is created, but they are not quite equal in length, and, more importantly, the colors are not evenly distributed.

\begin{figure}[tbp]
\centering
\includegraphics[width=\columnwidth]{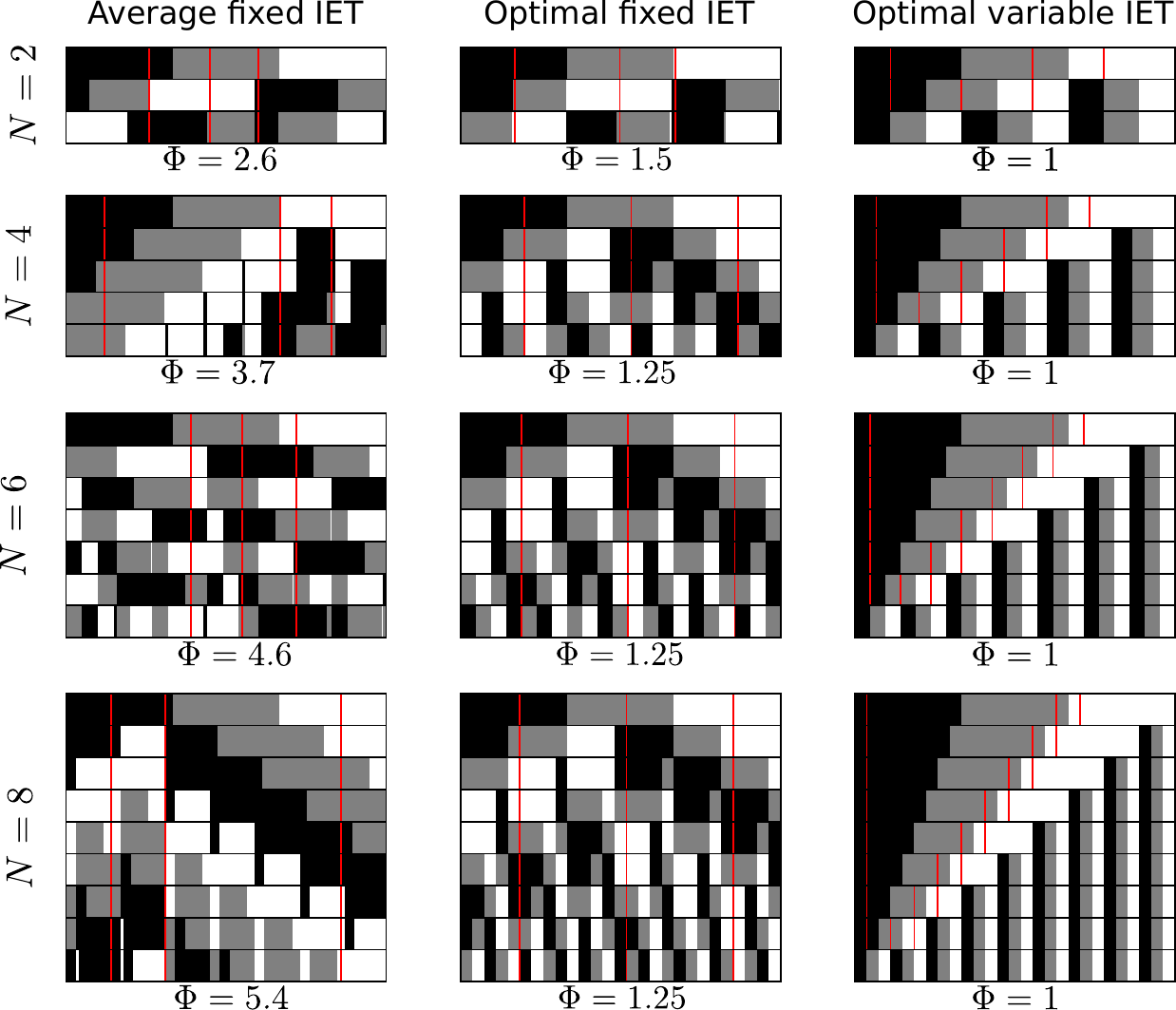}
\caption{Space-time plots for IETs with the three-color initial condition using the permutation $\pi=2413$ for $N=2$, $4$, $6$, and $8$. Cut locations are shown by red lines. First column: fixed IETs with $\Phi\approx\Phi_{\text{ave}}$. Second column: optimal fixed IETs with $\Phi=\Phi_{\text{min}}$. Third column: optimal variable IETs ($\Phi=1$) obtained using the methods described in \cref{sec:multi-color_extension}.}
\label{fig:periodic_optima_3_colors}
\end{figure}


\begin{figure}[tbp]
\centering
\includegraphics[width=\columnwidth]{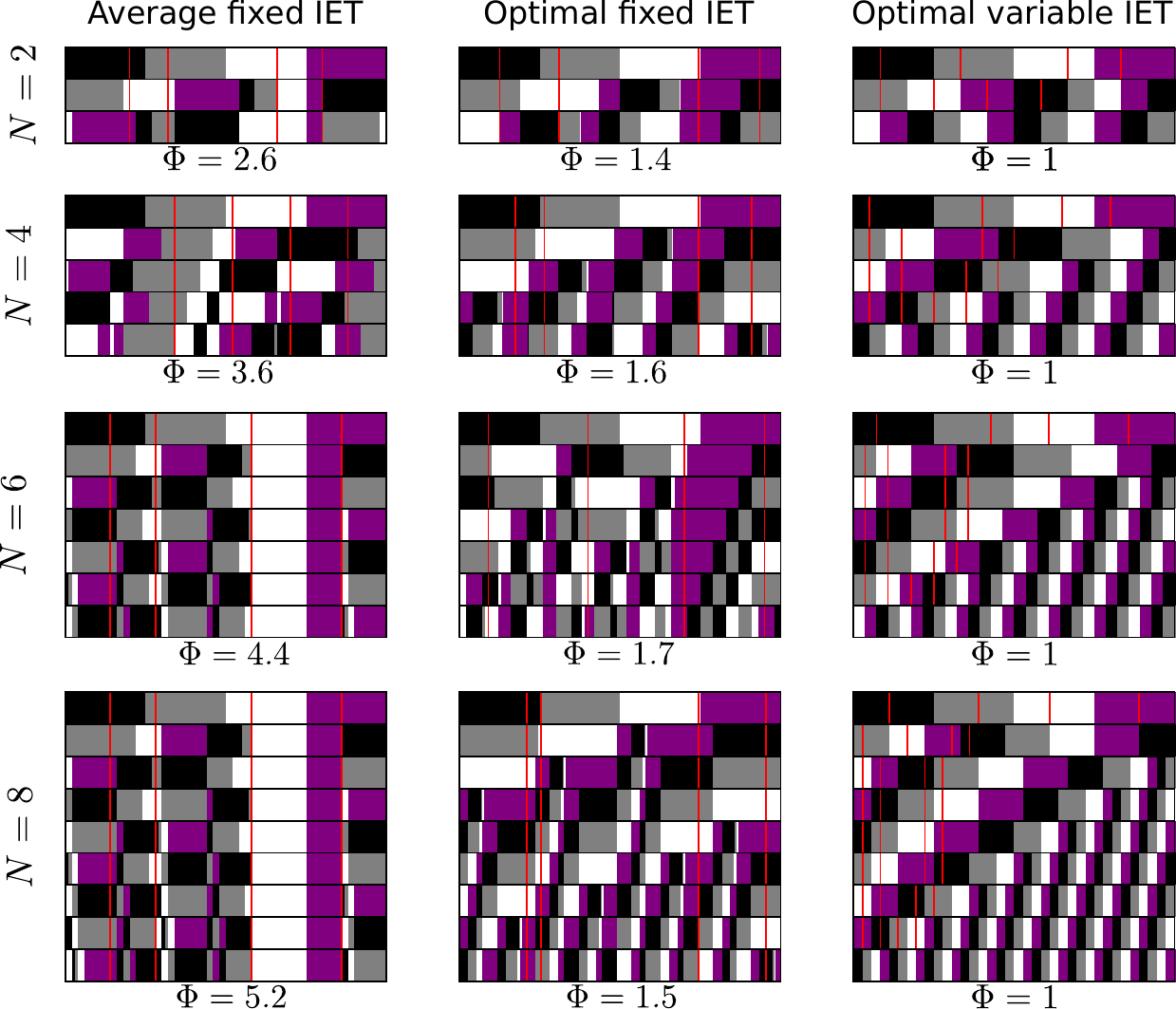}
\caption{Space-time plots for IETs with the four-color initial condition using the permutation $\pi=35241$ for $N=2$, $4$, $6$, and $8$. Cut locations are shown by red lines. First column: fixed IETs with $\Phi\approx\Phi_{\text{ave}}$. Second column: optimal fixed IETs with $\Phi=\Phi_{\text{min}}$. Third column: optimal variable IETs ($\Phi=1$) obtained using the methods described in \cref{sec:multi-color_extension}.
}
\label{fig:periodic_optima_4_colors}
\end{figure}

To summarize, increasing complexity in the initial condition (more colors) results in increased improvement in the mixing quality achieved by optimal variable IETs compared to both random fixed IETs and optimal fixed IETs. A simple explanation for this behavior is that there are more competing interests, i.e.\ competition between equal segment lengths and even distribution of the colors, when there are more colors in the initial condition. With more colors it becomes increasingly difficult to maintain a repeating fixed order of the colored segments without moving the cut locations. Using variable IETs overcomes these limitations in fixed IETs.

\section{Strategies for cutting-and-shuffling over many iterations in general systems} \label{sec:long-time_optimization}

While optima can be found analytically for variable IETs with simple initial conditions for arbitrary numbers of iterations, more complex cutting-and-shuffling systems require computationally expensive numerical optimization methods. For variable systems, this computational expense is compounded because the number of control parameters grows linearly with the number of iterations. In typical mixing problems, it is desired to optimize mixing over a small number of iterations, where numerical optimization methods are feasible. However, there may be situations where it is not possible to reach the desired mixing quality in a small number of iterations, and optimization over the required number of iterations is computationally prohibitive.

Even for fixed protocols, where the number of control parameters does not grow with the number of iterations, the distribution of the mixing metric across the parameter space is likely to be multi-modal, discontinuous, and complex, as demonstrated in \cref{fig:mix_metric_distros}(d) for the relatively simple case with the two-color initial condition, $L=3$, and only 8 iterations. This complexity generally increases with the number of iterations, making finding optima more challenging, as higher resolutions in the initial search grid are required.   

There are other general strategies for optimizing mixing over a large number of iterations that can be used for cutting-and-shuffling systems. The first is to use geometric properties of the piecewise isometry. For IETs, the Keane minimality condition can be used to predict long-term mixing quality: a fixed IET will be ergodic provided the permutation is irreducible and not a rotation, and the lengths of the cut pieces are rationally independent \cite{Keane1975, Viana2006}. Furthermore, fixed IETs satisfying the Keane minimality condition are almost always weak-mixing \cite{Avila2007}, which is a stronger sense of mixing than ergodicity \cite{Sturman2006}. Krotter \emph{et al.} \cite{Krotter2012} extended this work to find conditions for protocols to achieve good mixing in a finite time. These weak-mixing fixed IETs, while not optimal, can achieve good mixing over large numbers of iterations, i.e.\ in the limit of infinitely many iterations the domain will be completely homogenized. 

To explore this further, we use the same IET formulation as Krotter \emph{et al.}, which is demonstrated in \cref{fig:IET}. An irreducible permutation $\pi$ is chosen that is not a rotation, and a ratio of successive cut piece lengths, $r=|\mathcal{I}_{i+1}|/|\mathcal{I}_i|$, is chosen that determines the lengths of the cut pieces. For example, in \cref{fig:IET} the irreducible permutation $\pi=3142$ is used in all three examples, $r=1.5$ is used in \cref{fig:IET}(a,b), and $r=1+1/(2\pi)$ is used in \cref{fig:IET}(c). Here we examine this fixed IET approach over a large number of iterations for the two-color initial condition with a simpler irreducible permutation $\pi = 321$. We compare ratios $r=1+ 1/(2^i \pi)$ for $i=-1,0,1,2$, such that $r$ is irrational, and so the IET satisfies the Keane minimality condition. The dependence of the mixing metric $\Phi$ on $N$ for these cases is shown in \cref{fig:ad_hoc-vs-weak_mixing} (note the logarithmic scale of the vertical axis). Some of these examples produce relatively good mixing, with $\Phi =3.0$ at $N=100$, but for others, $\Phi$ is quite large, indicating poor mixing compared to the variable optimum ($\Phi=1$).

A similar strategy can be employed for 2D fixed PWIs using the exceptional set (where cuts occur), which is an intrinsic structure associated with 2D fixed PWIs. It has been shown that the area of the exceptional set can be used to predict the long-term mixing quality produced by 2D fixed PWIs \cite{Park2017}. Hence, long-term mixing can be improved by finding protocols that maximize the area of the exceptional set \cite{Smith2017BSTresonances}. As for IETs, over finitely many iterations this approach would likely yield poor mixing compared to the variable optimum.

Another strategy for mixing over large numbers of iterations is to optimize over short time-horizons, i.e.\ optimizing for every $m$ iterations, where $m$ is much smaller than the total number of iterations $N$. This strategy has proven effective for optimizing mixing in time-dependent fluid flows \cite{Cortelezzi2008}. In many practical applications, the difference between the optima obtained from short time-horizon optimization and the global optimum is likely to be insignificant. 

To demonstrate the effectiveness of short time-horizon optimization, we consider variable IETs with a time-horizon $m=1$, i.e.\ we optimize mixing at each iteration separately. The result is the \emph{ad hoc} method described briefly at the end of \cref{sec:metrics}, where the longest segments of each distinct color are cut in half at each iteration. For example, in the two-color case with permutation $132$ (two cuts within the domain), the longest black and longest gray segments are cut in half at each iteration, as shown in \cref{fig:ad_hoc_mixing}. If there are multiple segments that all have the longest length, then we can arbitrarily choose which one to cut. In \cref{fig:ad_hoc_mixing}, the cut is made in the first maximal black segment and in the first maximal gray segment when there are multiple segments that share the longest length. This \emph{ad hoc} approach is computationally inexpensive, and has a number of advantages compared to fixed IETs. First, this approach achieves optimal mixing ($\Phi=1$) whenever $N=2^i-1$, corresponding to points in \cref{fig:ad_hoc-vs-weak_mixing} where the dashed black curve touches the horizontal axis. Also, the worst mixing (maxima of $\Phi$) occurs at the iterations $N=2^i-2$, and it can be shown that at these iterates $\Phi=2-1/2^{i-1}$, so $\Phi<2$ for all $N$. Thus, the \emph{ad hoc} approach mixes significantly better than the weak-mixing protocols considered in \cref{fig:ad_hoc-vs-weak_mixing}. The key is that for fixed IETs, even weak-mixing ones, segments of the same color frequently reassemble. This means that the number of segments is generally significantly lower than the maximum, $N(L-1)+k$, that can be achieved by optimal variable IETs. On the other hand, reassembly of same-colored segments never occurs using the \emph{ad hoc} method, so the maximal number of segments is always achieved. The only limitation is that the segments do not have equal lengths.

Another advantage of the \emph{ad hoc} method, compared to both fixed IETs and optimal variable IETs, is that it is adaptive, i.e.\ it accounts for the current state of the scalar field. This means that if cuts are imprecise, as has been considered for fixed IETs like those in \cref{fig:IET} \cite{Yu2016}, the method can correct itself and not compound the error.

\begin{figure}[tbp]
\centering
\includegraphics[width=0.7\columnwidth]{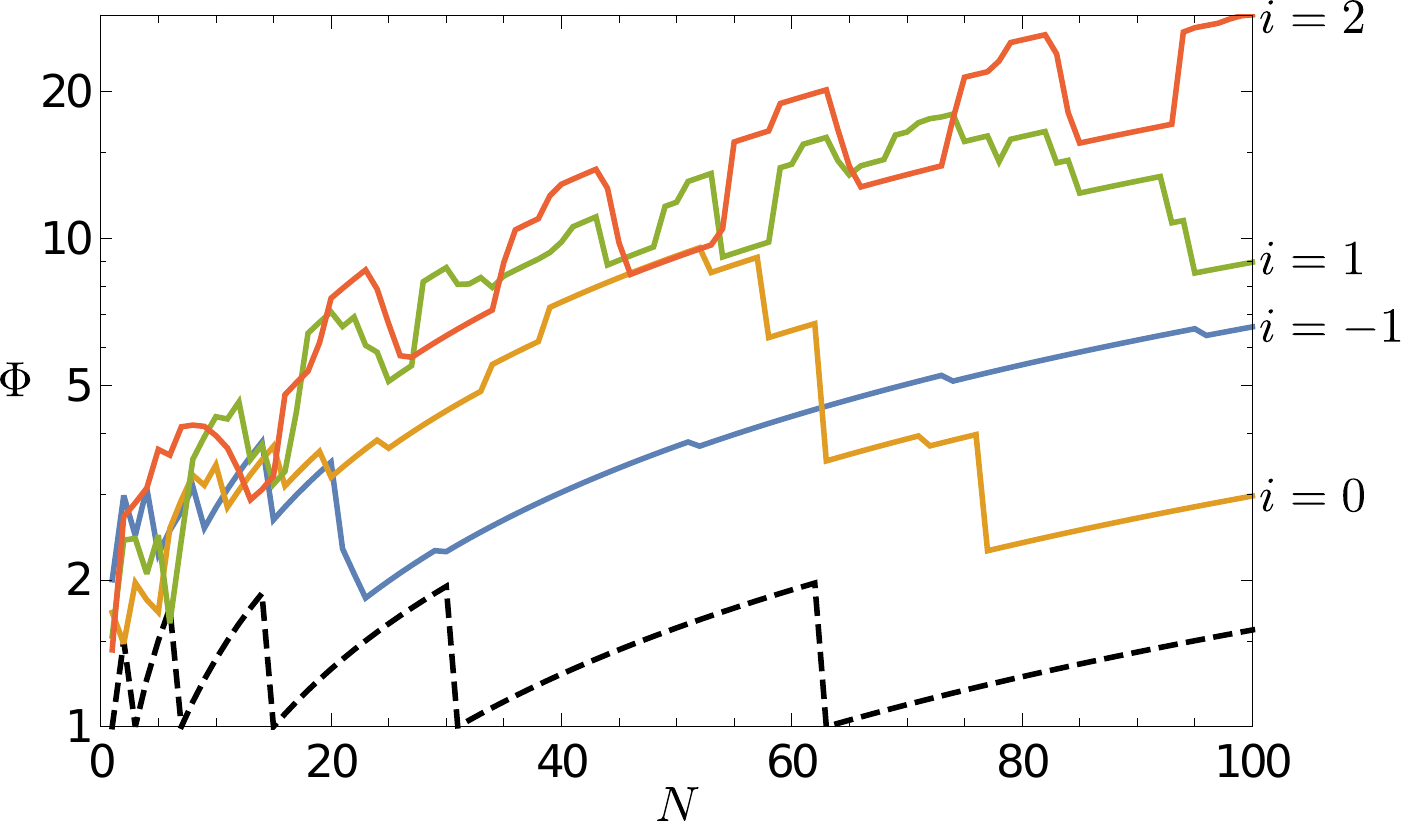}
\caption{Mixing with the \emph{ad hoc} method (dashed black), cutting the longest black and gray segments in half at each iteration and rearranging according to $\pi=132$, compared to four weak-mixing fixed IETs (solid) that use the irreducible permutation $\pi=321$ and satisfy the Keane minimality condition. For the weak-mixing fixed IETs, the same construction as Krotter \emph{et al.} is used, as demonstrated in \cref{fig:IET}, with $r=1+ 1/(2^i \pi)$ for $i=-1,0,1,2$. The optimal variable IET always results in $\Phi=1$ (the horizontal axis).}
\label{fig:ad_hoc-vs-weak_mixing}
\end{figure}

Therefore, the \emph{ad hoc} method provides a good alternative for mixing when the required number of iterations is large and finding optima for the total number of iterations is computationally prohibitive. Similar heuristics could be used for more general cutting-and-shuffling systems, such that cuts are located to bisect the largest unmixed regions.

\section{Conclusions} \label{sec:conclusions}

Variable cutting-and-shuffling strategies allow for significantly improved mixing compared to fixed cutting-and-shuffling. We have identified optimal variable cutting-and-shuffling strategies for IETs with an initial condition consisting of a number of differently colored segments, and we have shown that these optimal variable IETs can produce significantly better mixing than general fixed IETs. Furthermore, the improvement in mixing quality increases with the number of colors in the initial condition. This is because when there are more colors it is more difficult to satisfy the two competing interests: small colored segments and evenly distributed colors.

For more general systems with cutting-and-shuffling, optimizing mixing over large numbers of iterations, or when there are many cuts per iteration, is generally computationally prohibitive. This is especially true for variable strategies, where the number of control parameters increases linearly with the number of iterations. We demonstrate that an \emph{ad hoc} adaptive method, cutting the largest unmixed regions in half, provides a computationally inexpensive alternative. For IETs, this \emph{ad hoc} method is equivalent to optimizing using a one-iteration time-horizon, and yields significantly better mixing than examples of weak-mixing fixed IETs. Using the \emph{ad hoc} method, the mixing metric $\Phi$ is guaranteed to be within a constant factor of the optimum, for any number of iterations. The key to the success of the \emph{ad hoc} method is that segments of the same color never reassemble, meaning the maximum number of segments and interfaces will always be created. The adaptive nature of the \emph{ad hoc} method also means that it self corrects for any inexactness in the cut locations \cite{Yu2016}, or inexactness in the initial condition.


A key assumption when finding optimal mixing protocols is that the initial condition is known exactly. However, in many applications, the initial condition is not known exactly, but rather has a probability distribution. If the mixing quality is highly sensitive to the initial condition, then the optimal mixing protocol for a specific initial condition is irrelevant since this initial condition, or any other specific initial condition, may be unattainable in practice. In this case, it is more desirable to use a protocol that is optimal over a distribution of potential initial conditions. Future work should focus on this problem, developing strategies to optimize mixing when the initial condition has a probability distribution. This could be achieved by replacing the mixing metric with a weighted mean of the metric over the distribution of possible initial conditions. A yet more general approach could be to optimize mixing in a way that is entirely independent of the initial condition, for instance optimizing properties of the map itself. For example, maximizing the eigenvalues of the transfer operator in fluid flows maximizes the decay rate of any concentration field toward the uniform distribution \cite{Froyland2017}.

Another important question is how best to optimize mixing in more general variable systems with cutting-and-shuffling, including 2D piecewise isometries (PWIs) \cite{Scott2001, Scott2003, Haller1981, Park2016, Park2017, Juarez2010, Smith2017BSTresonances}, and systems with cutting-and-shuffling combined with stretching-and-folding and/or diffusion \cite{Ashwin2002, Sturman2012, Kreczak2017, Froyland2016, Smith2017BSTCS+SF, Smith2016discdef, Smith2017LSID, Smith2017CSSrate}. The added complexity means that exact methods like those in \cref{sec:analytic_optima} are unlikely to be possible, but numerical and heuristic optimization strategies could be developed. For example, symmetries could be used to systematically destroy non-mixing regions \cite{Franjione1989, Franjione1992}.



\appendix

\section{Color orders for cuts in optimal variable IETs} 
\label{app:color_order}

Consider the two-color initial condition, and a variable IET $T_{L,\pi}$ with permutation $\pi=132$ and $L=3$. For $T_{L,\pi}$ to produce the maximal number of new segments, in the first iteration one cut must be located within the black segment, and one cut located within the gray segment, as shown in \cref{fig:matching_edges}(a). Otherwise there would be an imbalance in the number of black and gray left and right edges of cut pieces. We also show in \cref{sec:analytic_optima} that in subsequent iterations, the maximum number of new segments will be created if we choose the first cut within a black segment and the second cut within a gray segment. However, it is not clear that this is the only choice that will work. What if we choose the first cut to be located within a gray segment, and the second cut located within a black segment? Consider the second iteration, where before cutting-and-shuffling the line, there are two black segments and two gray segments [\cref{fig:matching_edges}(b--e)]. Let $l_{2j}$ denote the colors of the left edges of the cut pieces, and $r_{2j}$ denote the colors of the right edges, as shown in \cref{fig:matching_edges}(b--e). Here we use $0$ to represent black and $1$ to represent gray. We know that $l_{21}=0$ because the left edge of the line is black, and $r_{23}=1$ because the right end of the line is gray. We show in \cref{sec:analytic_optima} that for optimal variable IETs, cuts must occur within segments, which translates to 
\begin{equation} \label{eq:edges_1}
r_{2j}=l_{2,j+1}, \quad j=1,2.
\end{equation}
After rearrangement, each black edge must join with a gray edge, which is expressed as 
\begin{equation} \label{eq:edges_2}
r_{2,\pi(j)} = l_{2,\pi(j+1)} +1 \!\! \mod 2, \quad j=1,2,3.
\end{equation}
Combining \cref{eq:edges_1}, \cref{eq:edges_2}, and the conditions $l_{21}=0$ and $r_{23}=1$, we have a linear system of equations for the variables $l_{2j}, r_{2j}$, $j=1,2,3$. Substituting \cref{eq:edges_1} into \cref{eq:edges_2} and using $l_{21}=0$, we reduce the system of equations to
\begin{equation} \label{eq:edges_reduced}
r_{2,\pi(j)} =
\begin{cases}
 r_{2,\pi(j+1)-1} + 1 \!\! \mod 2, & \text{if } \pi(j+1) \neq 1 \\
 1, & \text{if } \pi(j+1)=1,
 \end{cases}
\end{equation}
for $j=1,2,3$, so that $r_{21}$ and $r_{22}$ are the only two variables. Substituting $j=1,2,3$ into \cref{eq:edges_reduced}, we obtain
\begin{align} \label{eq:edges_reduced_2}
r_{21} &= r_{22} + 1, \nonumber \\ 
r_{23} &= r_{21} + 1, \\
r_{22} &= 1, \nonumber
\end{align}
where each equation is modulo $2$. This has the unique solution $r_{21}=0$, $r_{22}=1$. Hence, the first cut must occur within a black segment, and the second cut must occur within a gray segment. There are no other possibilities that will yield optimal mixing. Furthermore, \cref{eq:edges_1}--\cref{eq:edges_reduced_2} hold for all $N$, just replace $r_{2j}$ with $r_{Nj}$ and $l_{2j}$ with $l_{Nj}$, and so the first cut must always occur within a black segment, and the second cut within a gray segment.

Following the same procedure as above, \cref{eq:edges_reduced} must be satisfied for any variable IET $T_{L,\pi}$ that produces optimal mixing. We discuss two cases with $L=5$. First, let $\pi = 14253$, which satisfies \cref{prop:valid_perms}, and so can produce optimal mixing. Substituting $j=1,\dots,5$ into \cref{eq:edges_reduced}, we obtain
\begin{align} \label{eq:edges_reduced_3}
r_{N1} &= r_{N3} + 1, \nonumber \\ 
r_{N4} &= r_{N1} + 1, \nonumber \\
r_{N2} &= r_{N4} + 1, \\
r_{N5} &= r_{N2} + 1, \nonumber \\
r_{N3} &= 1, \nonumber
\end{align}
where each equation is modulo $2$. In addition, $r_{N5}=1$ because the right edge of the line must always be gray, since the left edge is always black [$\pi(1)=1$]. The system specified by \cref{eq:edges_reduced_3} has the unique solution $r_{N1}=r_{N2}=0$, $r_{N3}=r_{N4}=1$, which means the first two cuts must always occur within black segments, and the last two cuts must always occur within gray segments. 

Now consider $\pi = 14325$. Substituting $j=1,\dots,5$ into \cref{eq:edges_reduced}, we obtain
\begin{align} \label{eq:edges_reduced_4}
r_{N1} &= r_{N3} + 1, \nonumber \\ 
r_{N4} &= r_{N2} + 1, \nonumber \\
r_{N3} &= r_{N1} + 1, \\
r_{N2} &= r_{N4} + 1, \nonumber \\
r_{N5} &= 1, \nonumber
\end{align}
where each equation is modulo $2$. Note that the first and third equations are equivalent, as are the second and fourth. In addition, the last equation provides no new information. We already know that $r_{N5}=1$ because the right edge of the line must always be gray, since the left edge is always black [$\pi(1)=1$]. Therefore, for the unknown variables $r_{N1},\dots,r_{N4}$ we effectively have only two equations. The color at the first cut must be opposite to the color at the third cut, and likewise for the second and fourth cuts, but we are free to choose the colors of the third and fourth cuts. Hence there are four possible solutions:
\begin{align}
(r_{N3},r_{N4})=(0,0) &\Rightarrow (r_{N1},r_{N2})=(1,1), \nonumber \\
(r_{N3},r_{N4})=(0,1) &\Rightarrow (r_{N1},r_{N2})=(1,0), \\
(r_{N3},r_{N4})=(1,0) &\Rightarrow (r_{N1},r_{N2})=(0,1), \nonumber \\
(r_{N3},r_{N4})=(1,1) &\Rightarrow (r_{N1},r_{N2})=(0,0), \nonumber 
\end{align}
which means that at each iteration cuts can occur in the order gray, gray, black, black (GGBB), GBBG, BGGB, or BBGG, respectively. Therefore, there are many more optimal variable IETs that use the permutation $14325$ compared to the permutation $14253$.

This approach can be repeated for any permutation that satisfies \cref{prop:valid_perms}, and can also be extended to initial conditions with $k$ colors by simply changing \cref{eq:edges_2} to be modulo $k$.


\begin{thebibliography}{10}

\bibitem{Ashwin2002}
{\sc P.~Ashwin, M.~Nicol, and N.~Kirkby}, {\em Acceleration of one-dimensional
  mixing by discontinuous mappings}, Physica A, 310 (2002), pp.~347 -- 363,
  \url{https://doi.org/10.1016/S0378-4371(02)00774-4}.

\bibitem{Avila2007}
{\sc A.~Avila and G.~Forni}, {\em Weak mixing for interval exchange
  transformations and translation flows}, Ann. Math.,  (2007), pp.~637--664,
  \url{http://www.jstor.org/stable/20160038}.

\bibitem{Boujlel2016}
{\sc J.~Boujlel, F.~Pigeonneau, E.~Gouillart, and P.~Jop}, {\em Rate of chaotic
  mixing in localized flows}, Phys. Rev. Fluids, 1 (2016), p.~031301,
  \url{https://doi.org/10.1103/PhysRevFluids.1.031301}.

\bibitem{Christov2011}
{\sc I.~C. Christov, R.~M. Lueptow, and J.~M. Ottino}, {\em Stretching and
  folding versus cutting and shuffling: An illustrated perspective on mixing
  and deformations of continua}, Am. J. Phys., 79 (2011), pp.~359--367,
  \url{https://doi.org/10.1119/1.3533213}.

\bibitem{Christov2014}
{\sc I.~C. Christov, R.~M. Lueptow, J.~M. Ottino, and R.~Sturman}, {\em A study
  in three-dimensional chaotic dynamics: Granular flow and transport in a
  bi-axial spherical tumbler}, SIAM J. Appl. Dyn. Syst., 13 (2014),
  pp.~901--943, \url{https://doi.org/10.1137/130934076}.

\bibitem{Cortelezzi2008}
{\sc L.~Cortelezzi, A.~Adrover, and M.~Giona}, {\em Feasibility, efficiency and
  transportability of short-horizon optimal mixing protocols}, J. Fluid Mech.,
  597 (2008), p.~199–231, \url{https://doi.org/10.1017/S0022112007009676}.

\bibitem{Danckwerts1952}
{\sc P.~V. Danckwerts}, {\em The definition and measurement of some
  characteristics of mixtures}, Appl. Sci. Res., 3 (1952), pp.~279--296,
  \url{https://doi.org/10.1007/BF03184936}.

\bibitem{Franjione1989}
{\sc J.~G. Franjione, C.~Leong, and J.~M. Ottino}, {\em Symmetries within
  chaos: A route to effective mixing}, Phys. Fluids A-Fluid, 1 (1989),
  pp.~1772--1783, \url{https://doi.org/10.1063/1.857504}.

\bibitem{Franjione1992}
{\sc J.~G. Franjione and J.~M. Ottino}, {\em Symmetry concepts for the
  geometric analysis of mixing flows}, Philos. T. Roy. Soc. A, 338 (1992),
  pp.~301--323, \url{https://doi.org/10.1098/rsta.1992.0010}.

\bibitem{Froyland2016}
{\sc G.~Froyland, C.~González-Tokman, and T.~M. Watson}, {\em Optimal mixing
  enhancement by local perturbation}, SIAM Rev., 58 (2016), pp.~494--513,
  \url{https://doi.org/10.1137/15M1023221}.

\bibitem{Froyland2017}
{\sc G.~Froyland and N.~Santitissadeekorn}, {\em Optimal mixing enhancement},
  SIAM J. Appl. Math., 77 (2017), pp.~1444--1470,
  \url{https://doi.org/10.1137/16M1091496}.

\bibitem{Haller1981}
{\sc H.~Haller}, {\em Rectangle exchange transformations}, Monatsh. Math., 91
  (1981), pp.~215--232, \url{https://doi.org/10.1007/BF01301789}.

\bibitem{Hunt2008}
{\sc T.~P. Hunt, D.~Issadore, and R.~M. Westervelt}, {\em Integrated
  circuit/microfluidic chip to programmably trap and move cells and droplets
  with dielectrophoresis}, Lab. Chip, 8 (2008), pp.~81--87,
  \url{https://doi.org/10.1039/B710928H}.

\bibitem{Jones1988}
{\sc S.~W. Jones and H.~Aref}, {\em Chaotic advection in pulsed source-sink
  systems}, Phys. Fluids, 31 (1988), pp.~469--485,
  \url{https://doi.org/10.1063/1.866828}.

\bibitem{Juarez2012}
{\sc G.~Juarez, I.~C. Christov, J.~M. Ottino, and R.~M. Lueptow}, {\em Mixing
  by cutting and shuffling {3D} granular flow in spherical tumblers}, Chem.
  Eng. Sci., 73 (2012), pp.~195 -- 207,
  \url{https://doi.org/10.1016/j.ces.2012.01.044}.

\bibitem{Juarez2010}
{\sc G.~Juarez, R.~M. Lueptow, J.~M. Ottino, R.~Sturman, and S.~Wiggins}, {\em
  Mixing by cutting and shuffling}, Europhys. Lett., 91 (2010), p.~20003,
  \url{http://stacks.iop.org/0295-5075/91/i=2/a=20003}.

\bibitem{Katok1980}
{\sc A.~Katok}, {\em Interval exchange transformations and some special flows
  are not mixing}, Isr. J. Math., 35 (1980), pp.~301--310,
  \url{https://doi.org/10.1007/BF02760655}.

\bibitem{Keane1975}
{\sc M.~Keane}, {\em Interval exchange transformations}, Math. Z., 141 (1975),
  pp.~25--31, \url{https://doi.org/10.1007/BF01236981}.

\bibitem{Keane1977}
{\sc M.~Keane}, {\em Non-ergodic interval exchange transformations}, Isr. J.
  Math., 26 (1977), pp.~188--196, \url{https://doi.org/10.1007/BF03007668}.

\bibitem{Kreczak2017}
{\sc H.~Kreczak, R.~Sturman, and M.~C.~T. Wilson}, {\em Deceleration of
  one-dimensional mixing by discontinuous mappings}, Phys. Rev. E, 96 (2017),
  p.~053112, \url{https://doi.org/10.1103/PhysRevE.96.053112}.

\bibitem{Krotter2012}
{\sc M.~K. Krotter, I.~C. Christov, J.~M. Ottino, and R.~M. Lueptow}, {\em
  Cutting and shuffling a line segment: mixing by interval exchange
  transformations}, Int. J. Bifurcat. Chaos, 22 (2012), p.~1230041,
  \url{https://doi.org/10.1142/S0218127412300418}.

\bibitem{Leong}
{\sc C.~W. Leong and J.~M. Ottino}, {\em Experiments on mixing due to chaotic
  advection in a cavity}, J. Fluid Mech., 209 (1989), pp.~463--499,
  \url{https://doi.org/10.1017/S0022112089003186}.

\bibitem{Masur1982}
{\sc H.~Masur}, {\em Interval exchange transformations and measured
  foliations}, Ann. Math., 115 (1982), pp.~169--200,
  \url{https://doi.org/10.2307/1971341}.

\bibitem{Mathew2007}
{\sc G.~Mathew, I.~Mezi\'{c}, S.~Grivopoulos, U.~Vaidya, and L.~Petzold}, {\em
  Optimal control of mixing in stokes fluid flows}, J. Fluid Mech., 580 (2007),
  p.~261–281, \url{https://doi.org/10.1017/S0022112007005332}.

\bibitem{Mathew2005}
{\sc G.~Mathew, I.~Mezi{\'c}, and L.~Petzold}, {\em A multiscale measure for
  mixing}, Physica D, 211 (2005), pp.~23--46,
  \url{https://doi.org/10.1016/j.physd.2005.07.017}.

\bibitem{Mitchell2017}
{\sc R.~A. Mitchell and J.~D. Meiss}, {\em Designing a finite-time mixer:
  Optimizing stirring for two-dimensional maps}, SIAM J. Appl. Dyn. Syst., 16
  (2017), pp.~1514--1542, \url{https://doi.org/10.1137/16M1107139}.

\bibitem{Park2017}
{\sc P.~P. Park, T.~F. Lynn, P.~B. Umbanhowar, J.~M. Ottino, and R.~M.
  Lueptow}, {\em Mixing and the fractal geometry of piecewise isometries},
  Phys. Rev. E, 95 (2017), p.~042208,
  \url{https://doi.org/10.1103/PhysRevE.95.042208}.

\bibitem{Park2016}
{\sc P.~P. Park, P.~B. Umbanhowar, J.~M. Ottino, and R.~M. Lueptow}, {\em
  Mixing with piecewise isometries on a hemispherical shell}, Chaos, 26 (2016),
  073115, p.~073115, \url{https://doi.org/10.1063/1.4955082}.

\bibitem{Rallabandi2017}
{\sc B.~Rallabandi, C.~Wang, and S.~Hilgenfeldt}, {\em Analysis of optimal
  mixing in open-flow mixers with time-modulated vortex arrays}, Phys. Rev.
  Fluids, 2 (2017), p.~064501,
  \url{https://doi.org/10.1103/PhysRevFluids.2.064501}.

\bibitem{Schonfeld2004}
{\sc F.~Sch{\"o}nfeld, V.~Hessel, and C.~Hofmann}, {\em An optimised
  split-and-recombine micro-mixer with uniform ‘chaotic’mixing}, Lab. Chip,
  4 (2004), pp.~65--69, \url{https://doi.org/10.1039/B310802C}.

\bibitem{Schultz2014}
{\sc A.~Schultz, I.~Papautsky, and J.~Heikenfeld}, {\em {Investigation of
  Laplace Barriers for Arrayed Electrowetting Lab-on-a-Chip}}, Langmuir, 30
  (2014), pp.~5349--5356, \url{https://doi.org/10.1021/la500314v}.

\bibitem{Scott2003}
{\sc A.~Scott}, {\em Hamiltonian mappings and circle packing phase spaces:
  numerical investigations}, Physica D, 181 (2003), pp.~45 -- 52,
  \url{https://doi.org/10.1016/S0167-2789(03)00095-2}.

\bibitem{Scott2001}
{\sc A.~Scott, C.~Holmes, and G.~Milburn}, {\em Hamiltonian mappings and circle
  packing phase spaces}, Physica D, 155 (2001), pp.~34 -- 50,
  \url{https://doi.org/10.1016/S0167-2789(01)00263-9}.

\bibitem{Smith2017BSTresonances}
{\sc L.~D. Smith, P.~P. Park, P.~B. Umbanhowar, J.~M. Ottino, and R.~M.
  Lueptow}, {\em Predicting mixing via resonances: Application to spherical
  piecewise isometries}, Phys. Rev. E, 95 (2017), p.~062210,
  \url{https://doi.org/10.1103/PhysRevE.95.062210}.

\bibitem{Smith2016discdef}
{\sc L.~D. Smith, M.~Rudman, D.~R. Lester, and G.~Metcalfe}, {\em Mixing of
  discontinuously deforming media}, Chaos, 26 (2016), p.~023113,
  \url{https://doi.org/10.1063/1.4941851}.

\bibitem{Smith2017CSSrate}
{\sc L.~D. Smith, M.~Rudman, D.~R. Lester, and G.~Metcalfe}, {\em Impact of
  discontinuous deformation upon the rate of chaotic mixing}, Phys. Rev. E, 95
  (2017), p.~022213, \url{https://doi.org/10.1103/PhysRevE.95.022213}.

\bibitem{Smith2017LSID}
{\sc L.~D. Smith, M.~Rudman, D.~R. Lester, and G.~Metcalfe}, {\em Localized
  shear generates three-dimensional transport}, Chaos, 27 (2017), p.~043102,
  \url{https://doi.org/10.1063/1.4979666}.

\bibitem{Smith2017BSTCS+SF}
{\sc L.~D. Smith, P.~B. Umbanhowar, J.~M. Ottino, and R.~M. Lueptow}, {\em
  Mixing and transport from combined stretching-and-folding and
  cutting-and-shuffling}, Phys. Rev. E, 96 (2017), p.~042213,
  \url{https://doi.org/10.1103/PhysRevE.96.042213}.

\bibitem{Sturman2012}
{\sc R.~Sturman}, {\em The role of discontinuities in mixing}, Adv. Appl.
  Mech., 45 (2012), pp.~51--90,
  \url{https://doi.org/10.1016/B978-0-12-380876-9.00002-1}.

\bibitem{Sturman2008}
{\sc R.~Sturman, S.~W. Meier, J.~M. Ottino, and S.~Wiggins}, {\em Linked twist
  map formalism in two and three dimensions applied to mixing in tumbled
  granular flows}, J. Fluid Mech., 602 (2008), p.~129–174,
  \url{https://doi.org/10.1017/S002211200800075X}.

\bibitem{Sturman2006}
{\sc R.~Sturman, J.~M. Ottino, and S.~Wiggins}, {\em The mathematical
  foundations of mixing: the linked twist map as a paradigm in applications:
  micro to macro, fluids to solids}, Cambridge Monographs on Applied and
  Computational Mathematics (22), Cambridge University Press, 2006.

\bibitem{Veech1978}
{\sc W.~A. Veech}, {\em Interval exchange transformations}, J. Anal. Math., 33
  (1978), pp.~222--272, \url{https://doi.org/10.1007/BF02790174}.

\bibitem{Viana2006}
{\sc M.~Viana}, {\em Ergodic theory of interval exchange maps}, Rev. Mat.
  Complut., 19 (2006), pp.~7--100, \url{http://eudml.org/doc/40876}.

\bibitem{Vikhansky2002}
{\sc A.~Vikhansky}, {\em Enhancement of laminar mixing by optimal control
  methods}, Chem. Eng. Sci., 57 (2002), pp.~2719 -- 2725,
  \url{https://doi.org/10.1016/S0009-2509(02)00122-7}.

\bibitem{Yu2016}
{\sc M.~Yu, P.~B. Umbanhowar, J.~M. Ottino, and R.~M. Lueptow}, {\em Cutting
  and shuffling of a line segment: Effect of variation in cut location}, Int.
  J. Bifurcat. Chaos, 26 (2016), p.~1630038,
  \url{https://doi.org/10.1142/S021812741630038X}.

\bibitem{Zaman2017}
{\sc Z.~Zaman, M.~Yu, P.~P. Park, P.~B. Umbanhowar, J.~M. Ottino, and R.~M.
  Lueptow}, {\em Persistent structures in a {3D} dynamical system with solid
  and fluid regions}, arXiv preprint, arXiv:1706.00864.

\end{thebibliography}

\end{document}